\documentclass[11pt]{article}
\usepackage{amsmath,amsthm}
\usepackage{amssymb,mathrsfs,pifont}
\usepackage{bm,mathtools}

\usepackage{xcolor}
\usepackage{geometry}
\usepackage[colorlinks=true,
linkcolor=blue,citecolor=blue,
urlcolor=blue]{hyperref}

\makeatletter
\def\@seccntDot{.}
\def\@seccntformat#1{\csname the#1\endcsname\@seccntDot\hskip 0.5em}
\renewcommand\section{\@startsection{section}{1}{\z@}%
{18\p@ \@plus 6\p@ \@minus 3\p@}%
{9\p@ \@plus 6\p@ \@minus 3\p@}%
{\large\bfseries\boldmath}}
\renewcommand\subsection{\@startsection{subsection}{2}{\z@}%
{15\p@ \@plus 6\p@ \@minus 3\p@}%
{6\p@ \@plus 6\p@ \@minus 3\p@}%
{\itshape}}
\renewcommand\subsubsection{\@startsection{subsubsection}{3}{\z@}%
{12\p@ \@plus 6\p@ \@minus 3\p@}%
{\p@}%
{}}
\makeatother

\usepackage{microtype}

\theoremstyle{plain}
\newtheorem{theorem}{Theorem}[section]
\newtheorem{lemma}{Lemma}[section]

\newtheorem{corollary}{Corollary}[section]
\newtheorem{proposition}{Proposition}[section]

\theoremstyle{definition}

\newtheorem{remark}{Remark}[section]

\numberwithin{equation}{section}
\allowdisplaybreaks
\DeclareMathOperator{\ex}{ex}

\title{\bf Spectral Tur\'an Problems for Expanded hypergraphs}

\author {Zhenyu Ni$^{1}$, Dongquan Cheng$^{1}$,\,Jing Wang$^{2,3}$,  \,  Liying Kang$^{4}$\thanks{\em Corresponding author. Email address: lykang@shu.edu.cn (L. Kang), 995264@hainanu.edu.cn(Z. Ni), wj517062214@163.com(J. Wang), cdongquan@foxmail.com(D. Cheng).}\\
{\small $^{1}$Department of Mathematics, Hainan University,
Haikou 570228, P.R. China}\\
{\small $^{2}$School of Mathematics and Statistics, Henan Normal University,}
\\{\small Xinxiang 453007, P.R. China}\\
{\small $^3$Extremal Combinatorics and Probability Group, Institute for Basic Science}, \\{\small Daejeon, South Korea}\\
{\small $^{4}$Department of Mathematics, Shanghai University,
Shanghai 200444, P.R. China}}

\date{}

\begin{document}
\maketitle
\begin{abstract}

Given a graph $F$, the expansion $F^{(r)}$ of $F$  is defined as the $r$-uniform hypergraph obtained from $F$ by adding a set of $(r-2)$ distinct new vertices to each edge of $F$. In this paper, we investigate spectral stability results for hypergraphs and their applications. We first establish a spectral stability property: for any $r$-uniform hypergraph containing no copy of the expansion $F^{(r)}$ of a $(k+1)$-chromatic graph $F$, if  its $p$-spectral is close to the extremal value, then  the hypergraph is structurally close to $T_r(n, k)$, the complete $k$-partite $r$-uniform hypergraph  on $n$ vertices where sizes of any two parts differ by at most one. Using this spectral stability result, we determine
 the unique extremal hypergraph that maximizes the $p$-spectral radius among all $n$-vertex $r$-uniform hypergraphs without $t$ vertex-disjoint copies of the expansion $K_{k+1}^{(r)}$ of $K_{k+1}$. We prove that this extremal hypergraph is isomorphic to $K_{t-1}^{r} \,\vee\, T_r(n-t+1, k)$,  the join of the complete $r$-uniform hypergraph $K_{t-1}^{r}$ and $T_r(n-t+1, k)$. As a corollary, we  show  that $K_{t-1}^{r} \,\vee\, T_r(n-t+1, k)$ is the unique extremal hypergraph for $tK_{k+1}^{(r)}$,  which extends a result of Pikhurko [J. Combin. Theory Ser. B, 103 (2013) 220--225] for expanded complete graphs.


\par\vspace{2mm}

\par\vspace{2mm}

\end{abstract}

\section{Introduction}
As a natural generalization of graphs, hypergraphs offer a versatile framework for modeling higher-order relations in combinatorics and its adjacent fields.
In classical extremal hypergraph theory, the primary focus is on the Tur\'an number $\mathrm{ex}_r(n,\mathcal{F})$ for a given forbidden
$r$-uniform hypergraph
 $\mathcal{F}$. Specifically, $\mathrm{ex}_r(n,\mathcal{F})$ denotes the maximum number of edges in an $\mathcal{F}$-free $r$-uniform hypergraph on $n$ vertices.
More recently, inspired by the remarkable success of spectral methods in graph theory, research attention has shifted toward the  spectral Tur\'an problems for hypergraphs. In these problems, the edge count is replaced by an appropriate spectral parameter, and the central question becomes: what is the maximum possible spectral radius among all $\mathcal{F}$-free $r$-uniform hypergraphs on $n$ vertices?
Early studies on graphs (i.e., the case $r=2$) demonstrate that substituting the edge number with the spectral radius leads to a stronger version of the classical Tur\'an's theorem (see \cite{LiLiuFeng2022,Nikiforov2007,Nikiforov2010} for further details).

For any real number $p\geq 1$, the $p$-spectral radius of hypergraphs was introduced
by Keevash, Lenz and Mubayi \cite{Keevash-Lenz-Mubayi2014} and subsequently studied
by Nikiforov \cite{Nikiforov2014}. Given an $r$-uniform hypergraph (or $r$-graph)
 $\mathcal{H}$ of order $n$, and a vector
$\mathbf{x}=(x_1,x_2,\ldots,x_n)^{\mathrm{T}}\in\mathbb{R}^n$, denote
\[
P_{\mathcal{H}}(\mathbf{x}) = r! \sum_{e\in E(\mathcal{H})} \mathbf{x}^e,
\]
where $\mathbf{x}^{e}=\prod_{v\in e} x_v$.
The \emph{$p$-spectral radius} of $\mathcal{H}$ is defined as
\begin{equation}\label{eq:definition-p-spectral-radius}
	\lambda^{(p)}(\mathcal{H}):= \max_{\|\mathbf{x}\|_p=1} P_{\mathcal{H}}(\mathbf{x})
	= \max_{\|\mathbf{x}\|_p=1} r! \sum_{e\in E(\mathcal{H})} \mathbf{x}^e,
\end{equation}
where $\|\mathbf{x}\|_p:=(|x_1|^p+\cdots+|x_n|^p)^{1/p}$. If $\mathbf{x}\in\mathbb{R}^n$ is
a vector with $\|\mathbf{x}\|_p=1$ such that $\lambda^{(p)}(\mathcal{H})=P_{\mathcal{H}}(\mathbf{x})$, then $\mathbf{x}$
is called an \emph{eigenvector} corresponding to $\lambda^{(p)}(\mathcal{H})$. It is obvious that there is always
a nonnegative real vector $\mathbf{x}$ with $\|\mathbf{x}\|_p=1$ such that $\lambda^{(p)} (\mathcal{H}) = P_{\mathcal{H}}(\mathbf{x})$.
In the following, we always consider the nonnegative eigenvector.

By Lagrange's method, we have the eigenvalue-eigenvector equation for $\lambda^{(p)}(\mathcal{H})$
and eigenvector $\mathbf{x}$ as follows:
\begin{equation}\label{eq:eigenequation}
	\lambda^{(p)}(\mathcal{H})x_i^{p-1} = (r-1)! \sum_{e\in E(\mathcal{H}), i\in e}\mathbf{x}^{e\setminus \{i\}}
	~~\text{for}\ x_i>0.
\end{equation}


\begin{remark}
	It is worth mentioning that the $p$-spectral radius $\lambda^{(p)}(\mathcal{H})$ shows remarkable
	connections with some hypergraph invariants. For instance, the quantity $\lambda^{(1)}(\mathcal{H})$
	is the Lagrangian of $\mathcal{H}$, the quantity $\lambda^{(2)}(\mathcal{H})$ is the notion of hypergraph
	spectral radius introduced by Friedman and Wigderson \cite{Friedman-Wigderson1995},
	$\lambda^{(r)}(\mathcal{H})/(r-1)!$ is the usual spectral radius introduced by Cooper and Dutle \cite{Cooper2012},
	$\lambda^{(\infty)}(\mathcal{H})/r!$ is the number of edges of $\mathcal{H}$ (see \cite[Proposition 2.10]{Nikiforov2014}).
\end{remark}

Let $k\geq r\geq 2$. An $r$-graph $H$ is called \emph{$k$-partite} if its vertex set
$V(H)$ can be partitioned into $k$ sets so that each edge contains at most one vertex
from each set. An edge maximal $k$-partite $r$-graph is called \emph{complete $k$-partite}. For integers $n_1,\ldots,n_k$, we write $K_k^{r}(n_1,\ldots,n_k)$ for the
complete $k$-partite $r$-graph whose vertex classes have sizes
$n_1,\ldots,n_k$, respectively.
Let $T_r(n,k)$ denote the complete $k$-partite $r$-graph on $n$ vertices where no
two parts differ by more than one in size. We write $t_r(n,k)$ for the number of edges
of $T_r(n,k)$. That is,
\begin{equation}\label{eq:size-T(n,k)}
t_r(n,k) = \sum_{S\in\binom{[k]}{r}} \prod_{i\in S} \bigg\lfloor \frac{n+i-1}{k}\bigg\rfloor
= \big(1 - O(n^{-1})\big) \cdot \frac{(k)_r}{k^r} \binom{n}{r},
\end{equation}
where $(k)_r = k(k-1)\cdots(k-r+1)$.
Let $K_k^{r}$ denote the complete $r$-graph on $k$ vertices.

Given a graph $F$, the \emph{expansion} $F^{(r)}$ of $F$ is the $r$-graph
obtained by enlarging each edge of $F$ with $(r-2)$ new vertices.
For $k\ge r+1$, let $\mathcal{H}_{k}^{(r)}$ be the family of all $r$-graphs $\mathcal{F}$ with at most $\binom{k}{2}$ edges such that some $k$-set $C$ (the \emph{core}) satisfies the property that every pair $\{u,v\}\subseteq C$ is contained in an edge of $\mathcal{F}$. Clearly $K_{k}^{(r)}\in \mathcal{H}_{k}^{(r)}$.
When $r=2$, this family consists only of the complete graph $K_k$, while for $r>2$ it contains several non-isomorphic $r$-graphs.
Given a graph $F$, we write $\mathcal{H}_{F}^{(r)}$ for the family of all $r$-graphs $\mathcal{H}$ with at most $e(F)$ edges such that every edge $uv\in E(F)$ lies in some edge of $\mathcal{H}$. In particular, the expansion $F^{(r)}$ belongs to $\mathcal{H}_{F}^{(r)}$.

In 2006, Mubayi \cite{Mubayi2006} considered the Tur\'an problem for $\mathcal{H}_{k+1}^{(r)}$.
He showed that $\ex (n, \mathcal{H}_{k+1}^{(r)}) = t_r(n,k)$ with the unique extremal
$r$-graph being $T_r(n,k)$. Mubayi \cite{Mubayi2006} further established structural stability of near-extremal
$\mathcal{H}_{k+1}^{(r)}$-free $r$-graphs. Using this stability property, Pikhurko \cite{Pikhurko2013}
later strengthened Mubayi's result to show that $\ex (n, K_{k+1}^{(r)}) = t_r(n,k)$ for
all sufficiently large $n$. Recently, for a nondegenerate $r$-graph $F$,   Hou et al. \cite{HL1} presented a general approach for determining the maximum
number of edges in an $n$-vertex $r$-graph that does not contain $t + 1$ vertex-disjoint copies of $F$.
For a broad class of degenerate hypergraphs $F$, Hou et al. \cite{HL2} presented
near-optimal upper bounds for $\ex (n, (t+1)F)$.
For more results on Tur\'an number for $r$-graphs, we refer
the reader to the surveys \cite{Keevash2011} and \cite{Mubayi-Verstraete2016}.

The study of spectral Tur\'an problems for $r$-graphs has attracted significant attention in recent years. Early progress was made by Keevash, Lenz, and Mubayi \cite{Keevash-Lenz-Mubayi2014}, who established two general criteria for the $p$-spectral radius of $r$-uniform hypergraphs
that can be applied to obtain a variety of spectral Tur\'an-type results. Extending the spectral Mantel's theorem to hypergraphs,
Ni, Liu,  and Kang \cite{NiLiuKang2024} determined the maximum $p$-spectral radius
of $\{F_4, F_5\}$-free $3$-graphs, and characterized the extremal hypergraph,
where $F_4 = \{abc, abd, bcd\}$ and $F_5 = \{abc, abd, cde\}$. Recently, Zheng,
Li, and Fan \cite{ZhengLiFan2024} established the maximum $p$-spectral radius
of $K_{k+1}^{(r)}$-free $r$-graphs by leveraging results from Keevash, Lenz, and
Mubayi \cite{Keevash-Lenz-Mubayi2014}, and further characterized the extremal
hypergraph attaining this bound. Hou, Liu, and Zhao \cite{hou2024criterion} established a simple criterion for degree-stability of hypergraphs. An immediate application of this result, combined with the general theorem by Keevash, Lenz, and Mubayi~\cite{Keevash-Lenz-Mubayi2014}, solves the $p$-spectral Tur\'an problems for a large class of hypergraphs.

Research has also been conducted on spectral Tur\'an-type problems over specific
classes of hypergraphs.
Gao, Chang, and Hou \cite{GaoChangHou2022}
investigated the spectral extremal problem for $K_{r+1}^{(r)}$-free $r$-graphs among linear
hypergraphs. They proved that the spectral radius of an $n$-vertex
$K_{r+1}^{(r)}$-free linear $r$-graph is no more than $n/r$ when $n$ is sufficiently large.
Generalizing Gao, Chang, and Hou's result, She, Fan, Kang, and Hou \cite{SheFanKangHou2023}
presented sharp (or asymptotic) bounds of the spectral radius of $F^{(r)}$-free linear
$r$-graphs by establishing the connection between the spectral radius of linear
hypergraphs and those of their shadow graphs.  Here, $F$ is a graph  with chromatic
number $k \geq r +1$. Fang, Gao, Chang, and Hou \cite{FangGaoChangHou2025}
presented some sharp bounds on the spectral radius of $K^{(r)}_{s,t}$
free linear $r$-graphs by establishing the connection between the spectral radius and
the number of walks in uniform hypergraphs.
Another relevant result,
due to Ellingham, Lu, and Wang \cite{EllinghamLuWang2022}, showed that the
$n$-vertex outerplanar $3$-graph of maximum spectral radius is the unique $3$-graph
whose shadow graph is the join of an isolated vertex and the path $P_{n-1}$.

Spectral stability results are powerful tools for solving spectral extremal problems,
which say roughly that a near-extremal (with respect to spectral radius) $n$-vertex $F$-free
graph must be structurally close to the extremal graphs. Such stability results are crucial
in resolving spectral Tur\'an-type problems.
In this paper, we first establish a general spectral Tur\'an-type stability result for expanded hypergraphs.

\begin{theorem}\label{F+spectral-stability}
Let $F$ be a graph with $\chi(F)=k+1$, $p>1$, and $k\ge r\ge 3$.
For every $\varepsilon>0$, there exist constants
$\delta=\delta(k,r,\varepsilon)>0$ and $n_0=n_0(k,r,\varepsilon)$ such that
the following holds for all $n>n_0$:
If $\mathcal{H}$ is an $n$-vertex $F^{(r)}$-free $r$-graph with
\[
\lambda^{(p)}(\mathcal{H})
  >\lambda^{(p)}(T_r(n,k)) - \delta n^{r(1-1/p)},
\]
then $\mathcal{H}$ is $\varepsilon n^{r}$-close to $T_r(n,k)$, i.e., $\mathcal H$ can be transformed
to $T_r(n,k)$ by adding and deleting at most $\varepsilon\binom{n}{r}$ edges.
\end{theorem}

Our second main theorem determines the unique $r$-graph that maximizes
the $p$-spectral radius among all $n$-vertex $tK_{k+1}^{(r)}$-free $r$-graphs.

\begin{theorem}\label{tH}
Let $p \ge r \ge 3$, $k \ge r$, and $t \ge 1$ be integers.
For sufficiently large $n$, suppose that $\mathcal{H}$ has the maximum
$p$-spectral radius among all $n$-vertex $tK_{k+1}^{(r)}$-free $r$-graphs.
Then
\[
\mathcal H \cong K_{t-1}^{r} \,\vee\, T_r(n-t+1, k).
\]
\end{theorem}

Note that for any fixed $r$-graph $\mathcal H$, the $p$-spectral radius $\lambda^{(p)} (\mathcal H)$
is continuous in $p$. Thus, by letting $p\to\infty$ in Theorem \ref{tH},
we directly obtain the following corollary.

\begin{corollary}
Let $\mathcal H$ be a $tK_{k+1}^{(r)}$-free $r$-graph on $n$ vertices.
Then $e(H) \leq e(K_{t-1}^{r} \,\vee\, T_r(n-t+1, k))$ for sufficiently large $n$, with equality if
and only if $H \cong K_{t-1}^{r} \,\vee\, T_r(n-t+1, k)$.

\end{corollary}

\section{Preliminaries}

In this section we introduce the notation used throughout the paper and compile several auxiliary results needed for the proofs of our main theorems.

Let $\mathcal{H} = (V, E)$ be an $r$-graph, and let $U \subseteq V$ be a subset of vertices. The \emph{vertex-induced subhypergraph} of $\mathcal{H}$ on $U$ is defined as the $r$-graph
$\mathcal{H}[U] = \bigl(U,\ \{e \in E:  e \subseteq U\}\bigr)$,
where the edge set consists of all edges of $\mathcal{H}$ that are entirely contained in $U$.
Let $F\subseteq E$   be a nonempty subset of edges of $\mathcal{H}$. The \emph{edge-induced subhypergraph} of  $\mathcal{H}$  by $F$   is the $r$-graph $\mathcal{H}[F]=(V(F), F),$  where    $V(F)=\bigcup_{e\in F} e$
is the set of vertices contained in edges of   $F$.
The \emph{shadow graph} of an $r$-graph $\mathcal{H}$, written $\partial(\mathcal{H})$, is the graph with
\[
V(\partial(\mathcal{H}))=V(\mathcal{H})
\quad\text{and}\quad
E(\partial(\mathcal{H}))=\{f:\ f\subseteq e \text{ for some } e\in\mathcal{H}\}.
\]
If two $r$-graphs $\mathcal{H}_1=(V_1,E_1)$ and $\mathcal{H}_2=(V_2,E_2)$ have disjoint vertex sets, their disjoint \emph{union} is the hypergraph with vertex set $V_1\cup V_2$ and edge set $E_1\cup E_2$. We denote by $t\mathcal{H}$ the disjoint union of $t$ copies of $\mathcal{H}$.
The \emph{join} $\mathcal{H}_1\vee\mathcal{H}_2$ is obtained from the disjoint union of $\mathcal{H}_1$ and $\mathcal{H}_2$
by adding all $r$-edges that intersect both $V_1$ and $V_2$.
For two $r$-graphs $\mathcal{H}_1=(V,E_1)$ and $\mathcal{H}_2=(V,E_2)$ on the same vertex set, we define their sum $\mathcal{H}_1+\mathcal{H}_2$ as the $r$-graph with vertex set $V$ and edge set $E_1\cup E_2$.
For convenience, we write $\mathcal{H}_1+f$ instead of $\mathcal{H}_1+\mathcal{H}_2$ if $E_2=\{f\}$.

We first recall a subadditivity inequality due to Nikiforov \cite{Nikiforov2014} for the $p$-spectral radius.

\begin{proposition}[\cite{Nikiforov2014}]\label{prop:weyl-inequality}
Let $\mathcal{H}_1$ and $\mathcal{H}_2$ be edge-disjoint $r$-graphs on the same vertex set. Then
\[
\lambda^{(p)}(\mathcal{H}_1+\mathcal{H}_2)
 \le \lambda^{(p)}(\mathcal{H}_1)+\lambda^{(p)}(\mathcal{H}_2).
\]
\end{proposition}
The following bound generalizes the classical inequality $\lambda^{(2)}(G)\le\sqrt{2e(G)}$ to uniform hypergraphs.

\begin{lemma}[\cite{Keevash-Lenz-Mubayi2014}]\label{lem:size-upper-bound}
Let $p>1$ and let $\mathcal{H}$ be an $r$-graph with $m$ edges. Then
\[
\lambda^{(p)}(\mathcal{H})\le (r!\,m)^{1-1/p}.
\]
\end{lemma}

For convenience, we record the following elementary intersection inequality.

\begin{lemma}[\cite{CioabaFengTaitZhang2020}]\label{capset}
If $A_1,\dots,A_p$ are finite sets, then
\[
|A_1\cap\cdots\cap A_p|
 \ge \sum_{i=1}^p |A_i|-(p-1)\Bigl|\bigcup_{i=1}^p A_i\Bigr|.
\]
\end{lemma}

Let $G$ be a simple graph with matching number $\beta(G)$ and maximum degree $\Delta(G)$. For two integers $\beta$ and $\Delta$, define $f(\beta,\Delta)=\max\{e(G): \beta(G)\leq \beta,\Delta(G)\leq \Delta \}$. We next recall a result due to   Chv\'atal and Hanson \cite{ChvatalHanson} concerning graphs with bounded matching number and maximum degree.

\begin{lemma}[\cite{ChvatalHanson}]\label{matching-delta-function}
For integers $\beta\ge1$ and $\Delta\ge1$,
\[
f(\beta,\Delta)
  =\Delta\beta
   +\left\lfloor\frac{\Delta}{2}\right\rfloor
    \left\lfloor\frac{\beta}{\lceil\Delta/2\rceil}\right\rfloor
  \le \Delta\beta+\beta.
\]
\end{lemma}

We will use  the graph removal lemma  due to R{\"o}dl and Skokan \cite{RodlSkokan2006}.

\begin{lemma}[\cite{RodlSkokan2006}]\label{removal-lemma}
For every $\varepsilon>0$ and every graph $F$ with $\chi(F)=k+1$,
there exists $n_0=n_0(\varepsilon,F)$ such that any $F$-free graph $G$ on $n>n_0$ vertices
can be made $K_{k+1}$-free by removing at most $\varepsilon n^2$ edges.
\end{lemma}

A Perron--Frobenius type theorem for the $p$-spectral radius, due to Nikiforov \cite{Nikiforov2014},
will  be needed.

\begin{theorem}[\cite{Nikiforov2014}]\label{PFThm}
Let $\mathcal{H}$ be an $r$-graph and let $\mathbf{x}$ be an eigenvector corresponding to $\lambda^{(p)}(\mathcal{H})$.
\begin{itemize}
    \item[(1)] If $p>r-1$ and $\mathcal{H}$ is connected, then $x_i>0$ for all $i\in[n]$.
    \item[(2)] If $p>r$, then $x_i>0$ for each non-isolated vertex of $\mathcal{H}$.
\end{itemize}
\end{theorem}
We also need two inequalities given by Kang, Nikiforov, and Yuan \cite{Kang-Nikiforov-Yuan2014}.
\begin{lemma}[{\cite{Kang-Nikiforov-Yuan2014}}]\label{lem:twoinequalities}
	Let $n_1, \cdots,n_k$ be $k$ integers, $y_1,\cdots,y_k$ be $k$ positive reals and $p>\frac{9}{8}$.
	If $n_1\ge \cdots\ge n_k$, $n_1y_1^p+\cdots+n_ky_k^p=1$ and $n_1\ge n_k+2$,
	then $n_1y_1+n_ky_k\le (n_1+n_k)y^*$ and
$n_1n_ky_1y_k<\left\lceil\frac{n_1+n_k}{2}\right\rceil \left\lfloor\frac{n_1+n_k}{2}\right\rfloor (y^*)^2$, where $y^*=\left(\frac{n_1y_1^p+n_ky_k^p}{n_1+n_k}\right)^{\frac{1}{p}}$.
\end{lemma}

The following lemma compares the $p$-spectral radii of specific joined hypergraphs.
\begin{lemma}\label{lem:balance-set}
	Let $p>\frac{9}{8}$ and $n_1, \cdots,n_k$ be $k$ positive  integers satisfying $\sum_{i=1}^k n_i=n-t+1$.
	Then $$\lambda^{(p)}(K_{t-1}^{r}\vee K_k^{r}(n_1,\cdots,n_k))\le \lambda^{(p)}(K_{t-1}^{r}\vee T_r(n-t+1,k)).$$
\end{lemma}
\begin{proof}
	Let $\mathcal G=K_{t-1}^{r}\vee K_k^{r}(n_1,\cdots,n_k)$ have vertex partition  $V(\mathcal G)=U_0\cup U_1\cup \cdots \cup U_k$, where $|U_0|=t-1$ and $|U_i|=n_i$ for all $i\in [k]$. Without loss of generality,  we assume $n_1\ge \cdots\ge n_k$.
If $n_1\le n_k+1$,	the proof is complete. We thus assume
	 that $n_1\ge n_k+2$.	
Let  $m_1=\left\lceil\frac{n_1+n_k}{2}\right\rceil$, $m_k=\left\lfloor\frac{n_1+n_k}{2}\right\rfloor$,  $m_i=n_i$ for all $i\in[2,k-1]$, and let $\mathcal G' \cong K_{t-1}^{r}\vee K_k^{r}(m_1,\cdots,m_k)$ have vertex partition $V(\mathcal G')=V_0\cup V_1\cup \ldots \cup V_k$ such that $\mathcal G'[V_0]=K_{t-1}^{r}$, $V_1=U_1\setminus U'$, $V_k=U_k\cup U'$ and $V_i=U_i$ for any $i\in \{0,2,3,\ldots, k-1\}$, where $U'$ is a subset of $U_1$ with size $\left\lfloor\frac{n_1-n_k}{2}\right\rfloor$. It follows that $V_1\cup V_k=U_1\cup U_k$.

It  suffices to prove that $\lambda^{(p)}(\mathcal G)< \lambda^{(p)}(\mathcal G')$.
Let $\mathbf{y}$ be the nonnegative eigenvector satisfying $P_{\mathcal G}(\mathbf{y})=\lambda^{(p)}(\mathcal G)$.
	By symmetry, we may assume that $y_u=y_i$ for all $u\in U_i, i\in \{0,1,2,\cdots,k\}$, with the normalization condition $(t-1)y_0^p+n_1y_1^p+\cdots+n_ky_k^p=1$.
Let $y^* = \left(\frac{n_1y_1^p+n_ky_k^p}{n_1+n_k}\right)^{\frac{1}{p}}$, and define an $n$-dimensional vector $\mathbf{z}$ by setting  $z_u=y^*$ for all $u\in U_1\cup U_k$ and $z_u=y_i$ for all $u\in U_i, i\in \{0,2,3,\cdots,k-1\}$.  It then follows that $\lambda^{(p)}(\mathcal G')\geq P_{\mathcal G'}(\mathbf{z})$.

We proceed to show that $P_{\mathcal G'}(\mathbf{z})>P_{\mathcal G}(\mathbf{y})=\lambda^{(p)}(\mathcal G)$, which will complete the proof.
	Let $E_1$ denote set of edges in $\mathcal G$  intersecting $V_1\cup V_k$, and $E_2$  the set of edges in $\mathcal G'$  intersecting $U_1\cup U_k$. For $i=1,2$, define $E_{i1}:=\{e\in E_i\colon e\cap U_0\neq \emptyset\}$ and $E_{i2}:=E_i\setminus E_{i1}$. We then have
\[P_{\mathcal G'}(\mathbf{z})-P_{\mathcal G}(\mathbf{y})=r!\bigg(\sum_{e\in E_{21}}\mathbf z^{e}+\sum_{e\in E_{22}}\mathbf z^{e}-\sum_{e\in E_{11}}\mathbf y^{e}-\sum_{e\in E_{12}}\mathbf y^{e}\bigg).\]
By the constructions of $\mathcal G$ and $\mathcal G'$, $E_{11}=E_{21}$. For each $s\in[r-1]$, let $F(r-s)$ denote the family of $(r-s)$-sets of $U_0\cup U_2\cup \cdots \cup U_{k-1}$ that intersect $U_0$.  Every $s$-set of $U_1\cup U_k$ and every $(r-s)$-set $f\in F(r-s)$ together form an edge in $E_{11}$, which yields  the following identities:
	\begin{align*}
		\sum_{e\in E_{11}}\mathbf y^{e}=\sum_{s=1}^{r-1}\bigg(\sum_{i+j=s}\binom{n_1}{i}\binom{n_k}{j}y_1^{i}y_k^{j}\bigg)\bigg(\sum_{f\in F(r-s)}\mathbf y^{f}\bigg)
	\end{align*}
and
\begin{align*}
\sum_{e\in E_{21}}\mathbf z^{e}=\sum_{s=1}^{r-1}\binom{n_1+n_k}{s}(y^*)^{s}\bigg(\sum_{f\in F(r-s)}\mathbf y^{f}\bigg).
	\end{align*}

	Consider the $s$-uniform complete hypergraph $\mathcal K=K_{n_1+n_k}^{s}$, we have
	$$P_{\mathcal K}(\mathbf{y}|_{U_1\cup U_k})=r!\sum_{i+j=s}\binom{n_1}{i}\binom{n_k}{j}y_1^{i}y_k^{j}\ \mbox{and} \   P_{\mathcal K}(\mathbf{z}|_{U_1\cup U_k})=r!\binom{n_1+n_k}{s}(y^*)^{s}.$$
	Since $n_1y_1^{p}+n_ky_{k}^{p}=(n_1+n_k)(y^*)^{p}$,  the property of $p$-spectral radius of $s$-uniform complete hypergraphs implies that
 \[\sum_{e\in E_{11}}\mathbf y^{e}\le \sum_{e\in E_{21}}\mathbf z^{e}.\]

We now turn to the edges in $E_{12}$ and $E_{22}$.
For each $i\in[2,k-1]$, let $F(r-1)$ denote the family of $(r-1)$-set of $U_2\cup \cdots \cup U_{k-1}$ that intersect $U_i$ in at most one vertex, and $F(r-2)$ denote the  family of $(r-2)$-set of $U_2\cup \cdots \cup U_{k-1}$ that intersect $U_i$ in at most one vertex.
By the constructions of $\mathcal G$ and $\mathcal G'$,
 we have
	\begin{align*}
		\sum_{e\in E_{12}}\mathbf y^{e}=\big(n_1y_1+n_ky_k\big)\bigg(\sum_{f\in F(r-1)}\mathbf y^{f}\bigg)+n_1n_ky_1y_k\bigg(\sum_{f\in F(r-2)}\mathbf y^{f}\bigg)
	\end{align*}
	and
	\begin{align*}
		\sum_{e\in E_{22}}\mathbf z^{e}=(n_1+n_k)y^*\bigg(\sum_{f\in F(r-1)}\mathbf y^{f}\bigg)+\big\lceil\frac{n_1+n_k}{2}\big\rceil \big\lfloor\frac{n_1+n_k}{2}\big\rfloor (y^*)^2\bigg(\sum_{f\in F(r-2)}\mathbf y^{f}\bigg).
	\end{align*}
	By Lemma \ref{lem:twoinequalities}, it follows that $\sum_{e\in E_{12}}\mathbf y^{e}< \sum_{e\in E_{22}}\mathbf z^{e}$, which completes the proof.
	\end{proof}	

\section{Proof of Theorem \ref{F+spectral-stability}}
A fundamental stability theorem for expanded cliques was established by Mubayi \cite{Mubayi2006}, and a spectral generalization of this result was recently derived by Liu, Ni, Wang and Kang~\cite{LiuNiWangKang2024}.

\begin{lemma}[\cite{LiuNiWangKang2024}]\label{lem:spectral-stability-general-clique}
Let $p>1$ and $k\ge r\ge 3$.
For every $\varepsilon>0$ there exist constants $\delta=\delta(k,r,\varepsilon)>0$
and $n_0=n_0(k,r,\varepsilon)$ such that the following holds for all $n>n_0$:
If $\mathcal{H}$ is an $n$-vertex $\mathcal{H}_{k+1}^{(r)}$-free $r$-graph and
\[
\lambda^{(p)}(\mathcal{H})
  > \lambda^{(p)}(T_r(n,k)) - \delta n^{r(1-1/p)},
\]
then $\mathcal{H}$ is $\varepsilon n^{r}$-close to $T_r(n,k)$.
\end{lemma}

We extend the above result to the  expansion of  $(k+1)$\nobreakdash-chromatic graph $F$.
\begin{lemma}\label{lem:spectral-stability} 
Let $F$ be a graph with $\chi(F)=k+1$, $p>1$ and $k\geq r\geq 3$. For any $\varepsilon>0$, there are $\delta>0$ and $n_0$ such that the following holds for all $n>n_0$: If $\mathcal{H}$ is an $n$-vertex
$\mathcal{H}_{F}^{(r)}$-free $r$-graph with
\[\lambda^{(p)}(\mathcal{H})>\lambda^{(p)}(T_r(n,k))-\delta n^{r(1-1/p)},\]
then $\mathcal{H}$ is $\varepsilon n^{r}$-close to $T_r(n,k)$.
\end{lemma}
\begin{proof}
Given $\varepsilon>0$, by Lemma \ref{lem:spectral-stability-general-clique}, there exist $\delta_1>0$ and $n_0$ such that any $\mathcal{H}_{k+1}^{(r)}$-free $r$-graph $\mathcal{G}$ having $n>n_0$ vertices and $\lambda^{(p)} (\mathcal{G}) \geq  \lambda^{(p)}(T_r(n,k))-\delta_1 n^{r(1-1/p)}$ is $\frac{\varepsilon}{2}n^{r}$-close to $T_r(n, k)$.
Choose \[\varepsilon_1 =\min\left\{\frac{\varepsilon}{2},\frac{1}{r(r-1)}\cdot\big(\frac{\delta_1}{2}\big)^{p/(p-1)}\right\}, \ \mbox{and} \ \delta=\frac{\delta_1}{2}.\]

Let $G$ be the shadow graph  of $\mathcal{H}$. Since $\mathcal{H}$ is an $n$-vertex $\mathcal{H}_{F}^{(r)}$-free $r$-graph, $G$ is $F$-free on $n$ vertices.
Recall that $\chi(F)=k+1$, by Lemma \ref{removal-lemma}, there exists a set $E_0\subseteq E(G)$ with $|E_0|=\varepsilon_1 n^2$ such that the graph $G'$ obtained from $G$ by deleting edges of $E_0$ is $K_{k+1}$-free.
Let $\mathcal{H}'$ be the $r$-graph obtained from $\mathcal{H}$ by deleting all $e\in E(\mathcal{H})$  containing an edge in $f\in E_0$.
Then $e(\mathcal{H}\setminus \mathcal{H}')\leq \varepsilon_1 n^2 \binom{n}{r-2}\leq \frac{\varepsilon}{2}n^r$ and $\mathcal{H}'$ is $\mathcal{H}_{k+1}^{(r)}$-free.

Next, we give an estimation of $\lambda^{(p)} (\mathcal{H}')$.
Lemma \ref{lem:size-upper-bound} implies that
 \[
\lambda^{(p)} (\mathcal{H}\setminus \mathcal{H}') < \big(r! \cdot e(\mathcal{H}\setminus \mathcal{H}')\big)^{1-1/p} <
\frac{\delta_1}{2}  n^{r(1-1/p)}.
\]
Hence, by Proposition \ref{prop:weyl-inequality},
\begin{align*}
\lambda^{(p)} (\mathcal{H}')
& \geq \lambda^{(p)}(\mathcal{H}) - \lambda^{(p)} (H\setminus \mathcal{H}') \\
& > \lambda^{(p)}(T_r(n,k))-\frac{\delta_1}{2} n^{r(1-1/p)}-\frac{\delta_1}{2}  n^{r(1-1/p)} \\
& \geq  \lambda^{(p)}(T_r(n,k))-\delta_1 n^{r(1-1/p)}.
\end{align*}
It follows from Lemma \ref{lem:spectral-stability-general-clique} that
$\mathcal{H}'$ is $\frac{\varepsilon}{2} n^{r}$-close to $T_r(n,k)$.
Therefore, by the construction of $\mathcal{H}'$, $\mathcal{H}$ is $\varepsilon n^{r}$-close to $T_r(n,k)$.
\end{proof}

We are now ready to prove Theorem~\ref{F+spectral-stability}.

\noindent \textbf{Proof of Theorem \ref{F+spectral-stability}}.
Given $\varepsilon > 0$, by Lemma \ref{lem:spectral-stability}, there exist
$\delta_1>0$ and $n_0$ such that any $\mathcal{H}_{F}^{(r)}$-free $r$-graph $\mathcal{G}$ having $n\geq n_0$ vertices and
$\lambda^{(p)} (\mathcal{G}) > \lambda^{(p)}(T_r(n,k))-\delta_1 n^{r(1-1/p)}$ is $\frac{\varepsilon}{2} n^{r}$-close to $T_r(n,k)$.
Choose
\[
\delta = \min \bigg\{ \frac{\delta_1}{2}, \big(r!\cdot \frac{\varepsilon}{2}\big)^{1-1/p} \bigg\}.
\]
Let $|V(F)|=f$ and $d= \bigg(f + (r-2)\binom{f}{2}\bigg) \binom{n}{r-3}$.
Let $\mathcal{H}_1$ be the $r$-graph obtained from $\mathcal{H}$ by removing all edges that contain pairs with codegree at most $d$.
Since the number of pairs of vertices is $\binom{n}{2}$, we obtain
\begin{equation}\label{eq:upper-bound-H-H1}
\begin{split}
e(\mathcal{H}\setminus \mathcal{H}_1)
 \leq \bigg(f + (r-2)\binom{f}{2}\bigg) \binom{n}{r-3} \times \binom{n}{2}  < \frac{\delta^{p/(p-1)}}{r!} n^{r}.
\end{split}
\end{equation}

Now, we show that $\mathcal{H}_1$ is $\mathcal{H}_{F}^{(r)}$-free.
Suppose to the contrary that $\mathcal{H}_1$ contains a copy of some $\mathcal{G}\in \mathcal{H}_{F}^{(r)}$ with core $C$. Then every pair from $C$ is covered by an edge of $\mathcal{H}_1$, implying that
every pair in $C$ has codegree at least $d$ in $\mathcal{H}$.
Consequently,
we can greedily choose edges of $\mathcal{H}$ that contain all pairs in $\binom{C}{2}$,
such that these edges intersect $C$ in exactly two vertices and are pairwise disjoint
outside $C$.
Then $\mathcal{H}$ contains $F^{(r)}$ as a subgraph, a contradiction.

Next, we give an estimation of $\lambda^{(p)} (\mathcal{H}_1)$. By Proposition \ref{prop:weyl-inequality},
\[
\lambda^{(p)} (\mathcal{H}) \leq \lambda^{(p)} (\mathcal{H}_1) + \lambda^{(p)} (\mathcal{H}\setminus \mathcal{H}_1).
\]
On the other hand, Lemma \ref{lem:size-upper-bound} and \eqref{eq:upper-bound-H-H1} imply that

\[
\lambda^{(p)} (\mathcal{H}\setminus \mathcal{H}_1) < \big(r! \cdot e(\mathcal{H}\setminus \mathcal{H}_1)\big)^{1-1/p} <
\delta  n^{r(1-1/p)},
\]
where the last inequality holds due to \eqref{eq:upper-bound-H-H1}. Hence,
\begin{align*}
\lambda^{(p)} (\mathcal{H}_1)
& \geq \lambda^{(p)}(\mathcal{H}) - \lambda^{(p)} (\mathcal{H}\setminus \mathcal{H}_1) \\
& > \lambda^{(p)}(T_r(n,k))-\delta n^{r(1-1/p)}-\delta  n^{r(1-1/p)}  \\
& \geq \lambda^{(p)}(T_r(n,k))-\delta_1 n^{r(1-1/p)}.
\end{align*}
It follows from Lemma \ref{lem:spectral-stability} that
$\mathcal{H}_1$ is $\frac{\varepsilon}{2} n^{r}$-close to $T_r(n,k)$.
Recall that $\mathcal{H}_1$ is obtained from $\mathcal{H}$ by removing at most
\[
\frac{\delta^{p/(p-1)}}{r!} n^{r} \leq \frac{\varepsilon}{2} n^{r}
\]
edges by \eqref{eq:upper-bound-H-H1}. Hence, $\mathcal{H}$ is
$\varepsilon n^{r}$-close to $T_r(n,k)$. This completes the proof.
\qed

\section{Proof of Theorem \ref{tH}}

For the remainder of this section, we  assume that $k \ge r \ge 3$, $p \ge r$,  and that $\mathcal{H}$ is an $n$-vertex $tK_{k+1}^{(r)}$-free $r$-graph with the maximum $p$-spectral radius. We further let $\varepsilon>0$ be sufficiently small and $n$ sufficiently large to ensure all subsequent inequalities hold.

The sketch of our proof is as follows:
		Firstly, we start by showing that $\mathcal{H}$ is $\varepsilon n^{r}$-close to $T_r(n,k)$ via the
		spectral stability theorem for $F^{(r)}$ (Theorem \ref{F+spectral-stability}).
		Furthermore, we obtain an optimal partition $V_1,\ldots,V_k$ of $V(\mathcal{H})$ such that
		$|V_i| = (1 + o(1)) n/k$ for each $i\in [k]$, and define sparse pair, dense pair and dominant pair
		based on codegree conditions.
		Then we analyze the sets $L$ (vertices incident to many sparse pairs) and $W$ (vertices incident to many dominant pairs), and show that $|L| = o(n)$ and $|W \setminus L| \le t-1$ (Lemma~\ref{WL}).
Moreover, for each $i\in[k]$, we guarantee the existence of a large subset
$T_i \subseteq V_i \setminus (W\cup L)$ containing no dominant pairs (Lemma~\ref{independentset} ).
Next, we establish crucial degree estimates: for any vertex $u\in L$, $d_{\mathcal{H}}(u) < \binom{k-1}{r-1}(n/k)^{r-1} - 2\ell$
		(Lemma \ref{lem:degree-vertex-in-L}), while for the vertex $z$ with maximum eigenvector entry,
		$d_{\mathcal{H}}(z) > \binom{k-1}{r-1}(n/k)^{r-1} - 2\ell$ (Lemma \ref{lem:degree-of-u0}).
		We also bound the number of edges containing sparse pairs or dominant pairs.
		Finally, we prove $L = \emptyset$ (Lemma \ref{lem:L-emptyset}) and $|W| = t-1$ (Lemma \ref{lem:W=t-1}).
		This implies $\mathcal{H} \subseteq K_{t-1}^{r}\vee K_k^{r}(n_1,\cdots,n_k)$, where $|U_i|=n_i$  and $U_i=V_i\setminus W$ for $i\in [k]$.
		By the maximality of $\mathcal{H}$ and Theorem \ref{lem:balance-set}, we have $\mathcal{H} \cong  K_{t-1}^{r} \vee T_r(n-t+1,k)$

To prepare for our structural analysis, we first  introduce the canonical
$k$-partition associated with an $r$-graph.   Given an $r$-graph $\mathcal{G}$ of order $n$, let $\bm{\sigma} = (V_1,V_2,\ldots,V_k)$ denote
a partition of $V(\mathcal{G})$. We define

\[
f_{\mathcal{G}}(\bm{\sigma}) := \sum_{e\in E(\mathcal{G})} |\{i\in [k]: e\cap V_i \neq \emptyset\}|.
\]
The following lemma characterizes the approximate balance property of an optimal partition.
\begin{lemma}[\cite{LiuNiWangKang2024}]\label{lem:size-Vi}
	Let $k\geq r\geq 3$, and $\varepsilon > 0$. Suppose that $\mathcal{G}$ is an $r$-graph on $n$ vertices,
	and $\mathcal{G}$ is $\varepsilon n^{r}$-close to $T_r(n,k)$. Let $\bm{\sigma} = (V_1,V_2,\ldots,V_k)$
	be a partition of $V(\mathcal{G})$ such that $f_{\mathcal{G}}(\bm{\sigma})$ attains the maximum. Then for each $i\in [k]$,
	\[
	\Big( \frac{1}{k} - \varepsilon^{1/r} \Big) n \leq |V_i| \leq \Big(\frac{1}{k} + \varepsilon^{1/r} \Big) n.
	\]
\end{lemma}

Since $T_r(n,k)$ is $tK_{k+1}^{(r)}$-free, we have
\begin{eqnarray}\label{eq4.1}
\lambda ^{(p)}(\mathcal{H})\geq \lambda^{(p)}(T_r(n,k))\geq\frac{(r-1)!}{k^{r-1}} \bigg(1 - O\Big(\frac{1}{n}\Big)\bigg)
  \binom{k-1}{r-1} n^{r(1-1/p)}.
\end{eqnarray}
By Theorem \ref{F+spectral-stability}, $\mathcal{H}$ is $\varepsilon n^r$-close to $T_r(n,k)$.  Lemma \ref{lem:size-Vi} then implies that  $\mathcal{H}$ admits a partition
$\bm{\sigma}=(V_1,\dots,V_k)$ which maximizes $f_{\mathcal{H}}(\bm{\sigma})$ and satisfies
\begin{equation}\label{lem4.1}
    \Bigl(\frac{1}{k}-\varepsilon^{1/r}\Bigr)n
    \le |V_i| \le
    \Bigl(\frac{1}{k}+\varepsilon^{1/r}\Bigr)n
    \qquad \text{for all } i\in[k].
\end{equation}

For this   partition $\bm{\sigma}$, let
$\mathcal{T}$ denote the complete $k$-partite $r$-graph on $V(\mathcal{H})$
with vertex classes $V_1,\dots,V_k$.
We call edges in $\mathcal{T}\setminus\mathcal{H}$
\emph{missing edges}, and edges in $\mathcal{H}\setminus\mathcal{T}$ \emph{bad edges}.
Define
\[
h \;:=\; \bigl|V\bigl(tK_{k+1}^{(r)}\bigr)\bigr|
      \  \mbox{and}\ d \;:=\; h\,\binom{n}{\,r-3\,}.
\]
For two   vertices $u, v$  from two different vertex classes of
 $\bm{\sigma}$, the pair $\{u,v\}$ is said to be  \emph{sparse} if the codegree of $u$ and $v$ is at most $d$, and \emph{dense} otherwise.
Let $L$ be the set of vertices contained in at least $\varepsilon^{1/r^2} n$ sparse pairs of $\mathcal{H}$, and
let  $\mathscr{L}$  denote the set of all sparse pairs in $\mathcal{H}$.

By the definition of $f_{\mathcal{H}}(\bm{\sigma})$, we have
\[
r\left( e(\mathcal{H}) - \varepsilon\binom{n}{r} \right) \leq f_{\mathcal{H}}(\bm{\sigma})\leq r\cdot e(\mathcal{H}) - e(\mathcal{H}\setminus T),
\]
from which  it immediately follows that
\begin{equation}\label{eq4.3}
e(\mathcal{H}\setminus T) \leq r\varepsilon\binom{n}{r}.
\end{equation}

\begin{lemma}\label{lem:size-L}
	$|\mathscr{L}|< \varepsilon^{1/r}n^2/2, ~~ |L|<\varepsilon^{\frac{r-1}{r^2}}n$.
\end{lemma}
\begin{proof}
	We first derive an upper bound on $|\mathscr{L}|$.
Since	each missing edge contains at most $\binom{r}{2}$ sparse pairs, whereas each sparse pair
	belongs to at least
	$\binom{k-2}{r-2} (1/k - \varepsilon^{1/r})^{r-2} n^{r-2}-d$ missing edges by (\ref{lem4.1}),
	we have
	\[
	|\mathscr{L}|\left[\binom{k-2}{r-2} \Big(\frac{1}{k} - \varepsilon^{1/r}\Big)^{r-2} n^{r-2} - d\right]
	\leq \binom{r}{2} e(\mathcal{T}\setminus \mathcal{H}).
	\]
	Recall that $d = h \binom{n}{r-3}$. Using Bernoulli's inequality, we obtain
	\begin{equation}\label{eq:sec4-temporary}
		|\mathscr{L}| \left[\frac{1}{k^{r-2}} \binom{k-2}{r-2} (1 - kr\varepsilon^{1/r}) n^{r-2}\right] < \binom{r}{2} e(\mathcal{T}\setminus \mathcal{H}).
	\end{equation}
	Since $\mathcal{H}$ is $\varepsilon n^{r}$-close to $T_r(n,k)$, $e(\mathcal{H}) > t_r(n,k) - \varepsilon n^{r}$.
	It follows that $e(\mathcal{T}) < e(\mathcal{H}) + \varepsilon n^{r}$. Combining this with (\ref{eq4.3}), we have
	\begin{align}\label{eq:size-T-H}
		e(\mathcal{T}\setminus \mathcal{H}) = (e(\mathcal{T}) - e(\mathcal{H})) + e(\mathcal{H}\setminus \mathcal{T}) < (r+1) \varepsilon n^{r}.
	\end{align}
	By \eqref{eq:sec4-temporary} and \eqref{eq:size-T-H}, we deduce
	\begin{equation*}\label{eq:size-W}
		|\mathscr{L}| < \frac{(r + 1) k^{r-2} \binom{r}{2} n^{r} \varepsilon}{\binom{k-2}{r-2} (1 - kr\varepsilon^{1/r}) n^{r-2}}
		< \frac{1}{2} \varepsilon^{1/r} n^2.
	\end{equation*}
	Finally, recall that $L$ is the set of vertices which contained in at least $\varepsilon^{1/r^2} n$
	sparse pairs of $\mathcal{H}$. Hence, $|L| \varepsilon^{1/r^2} n \leq 2|\mathscr{L}| < \varepsilon^{1/r} n^2$, which
	yields that $|L| < \varepsilon^{\frac{r-1}{r^2}} n$.
\end{proof}
For vertices $u,v$ in the same vertex class $V_i$ ($i\in [k]$), the pair $\{u,v\}$ is called  \emph{dominant}  if the codegree of $u$ and $v$ is at least $d$.
Let $W$ be the set of vertices contained in at least $\theta n$ dominant pairs in $\mathcal{H}$, where $\theta >\varepsilon^{\frac{r-1}{r^2}}$.
For a vertex $u\in V$, let $E_u$  denote the edges containing $u$;
for a vertex $u\in V\setminus L$, let $DV_i(u)$ be the set of vertices in $V_i$ forming a dense pair with
 $u$. By the definition of $L$, $|DV_j(u)|\geq |V_j|-\varepsilon^{1/r^2}n$ for all distinct
 $i, j\in [k]$ and $u\in V_i\setminus L$.

\begin{lemma}\label{set}
For any $i\in [k]$ and any finite subset $S\subseteq V\setminus (L\cup V_i)$, there exist at least $h$ vertices in $V_i\setminus L$ that  form  a dense pair with every vertex in $S$.
\end{lemma}
\begin{proof}
Let $S=\{u_1, u_2, \ldots, u_s\}$. Since $S\subseteq V\setminus (L\cup V_i)$,  the definition of $L$ gives $|DV_i(u_j)|\geq |V_i|-\varepsilon^{1/r^2}n$ for all $u_j\in S, j\in [s]$.
By Lemma \ref{capset},
\begin{eqnarray*}
\left|\bigcap_{j=1}^{s}DV_{i}(u_j)  \setminus L\right|
\geq  s(|V_{i}|-\varepsilon^{1/r^2}n)-(s-1)(|V_{i}|-\varepsilon^{\frac{r-1}{r^2}}n)
>h.
\end{eqnarray*}
This completes the proof. \end{proof}
\begin{lemma}\label{extention}
Let $e$ be a bad edge of $\mathcal H$. If there exists some $i\in [k]$ such that  $u, v\in e\cap (V_i\setminus L)$ for vertices $u$ and $v$,
then $\mathcal H$ contains a copy $\mathcal F$ of $K_{k+1}^{(r)}$ with $e\subseteq E(\mathcal F)$.
\end{lemma}
\begin{proof}
Without loss of generality, assume  $u, v\in V_1$.
By Lemma \ref{set}, there exist vertices
 $w_{i}\in V_i$ for  $i=2, \ldots, k$ such that every pair  in $\{u, v, w_{2}, \ldots, w_{k}\}$, except for $\{u, v\}$, is a dense pair.
Thus $\partial (\mathcal H[u, v,w_{2},\ldots,w_{k}])\cong K_{k+1}$,  and we may  extend this subgraph to a copy of $K^{(r)}_{k+1}$.
\end{proof}

\begin{lemma}\label{WL}
$|W\setminus L|\leq t-1$.
\end{lemma}
\begin{proof}
Suppose for a  contradiction that there exist vertices $u_1,\ldots,u_t\in W\setminus L$.  Since $u_i, i\in [t]$ is contained in at least $\theta n$ dominant pairs and $|L|\leq \varepsilon^{\frac{r-1}{r^2}}n$, there exist  vertices $w_{1},\ldots w_{t}\notin L$ such that $\{u_i,w_i\}$ is a dominant pair for each $i\in [t]$.
By Lemmas \ref{set} and \ref{extention}, there exist $t$ copies of $K^{(r)}_{k+1}$ in $\mathcal H$.
Since every dense pair has codegree greater than $d$, we may choose the additional vertices in these extensions greedily to ensure the $t$ copies of $K^{(r)}_{k+1}$ are vertex-disjoint.
This contradicts the fact that $\mathcal H$ is $tK^{(r)}_{k+1}$-free,
so we conclude $|W\setminus L|\le t-1$.
\end{proof}
\begin{lemma}\label{degree}
For any vertex $u\in V_i\setminus (W\cup L)$, where $ i\in [k]$, there exist at most $h-1$ vertices in $V_i\setminus (W\cup L)$ that form a dominant pair with $u$.
\end{lemma}
\begin{proof}
Suppose for a contradiction that there exists a vertex $u\in V_1\setminus (W\cup L)$ such that  at least $h$ vertices in $ V_1\setminus (W\cup L)$  form a dominant pair with $u$.
Since $u\not\in W$, there exists at least one vertex $v\in V_1$ such that $\{u,v\}$ is not a dominant pair of $\mathcal{H}$.
It follows that there exists an $r$-set $f\subset V(\mathcal H)$ containing $\{u,v\}$ with $f\notin E(\mathcal{H})$.
By the maximality of $\mathcal{H}$, the $r$-graph $\mathcal{H}'=\mathcal{H}+f$ contains a copy of $tK_{k+1}^{(r)}$, denoted   $\mathcal F_0$, with $f\in E(\mathcal F_0)$.
Moreover, $\mathcal F_0$ contains a copy of $(t-1)K_{k+1}^{(r)}$, denoted $\mathcal F_1$, satisfying $u\notin V(\mathcal F_1)$.
By  the assumption of $u$, there exists a vertex $w\in V_1\setminus (W\cup L\cup V(\mathcal F_1)$  such that $\{u,w\}$ is a dominant pair in $\mathcal H$.
By  Lemmas \ref{set} and \ref{extention}, $\mathcal H$ contains  a copy of $K^{(r)}_{k+1}$, denoted $\mathcal F_2$, with $V(\mathcal F_2)\cap V(\mathcal F_1)=\emptyset$.
Hence $\mathcal F_1\cup \mathcal F_2\cong tK_{k+1}^{(r)}\subseteq\mathcal H$, a contradiction.
\end{proof}

Define $D_i=\{v\in V_i : \text{ there exists at least one } v'\in  V_i\setminus (W\cup L) \text{ such that } \{v,v'\}$ $ \text{ is dominant}\}$ and  $T_i=V_i\setminus (D_i\cup W\cup L)$.

\begin{lemma}\label{independentset}
For each $i\in [k]$, the set $V_i\setminus (W\cup L)$ contains at most $f(t-1,h-1)$ dominant pairs. Moreover, $|D_i|\leq 2f(t-1,h-1)$ and $|T_i|\geq |V_i\setminus (W\cup L)|-2f(t-1,h-1)$.
\end{lemma}
\begin{proof}
For each $i\in [k]$, let $G_i$ denote the simple graph on $V_i\setminus (W\cup L)$ whose edges correspond to dominant pairs in $\mathcal{H}$. Lemma \ref{degree} implies $\Delta(G_i)\leq h-1$ for each $i\in [k]$. We claim that $\beta(G_i)\leq t-1$ for each $i\in [k]$.
Suppose for a contradiction that there exist $2t$ distinct vertices $u_1,\ldots,u_t,$ $v_1\ldots, v_t\in V_1\setminus (W\cup L)$ such that $\{u_j,v_j\}$ is a dominant pair in $\mathcal{H}$ for each $j\in [t]$.
By Lemmas \ref{set} and \ref{extention}, $\mathcal H$ contains $t$ copies of $K^{(r)}_{k+1}$.
Furthermore, the codegree condition ensures that the additional vertices used in these extensions can be chosen
to be vertex-disjoint for
each $j$. It follows that $\mathcal H$ contains $t$ vertex-disjoint copies of $K^{(r)}_{k+1}$, a contradiction.
Hence $\beta(G_i)\le t-1$.

For each $i\in [k]$, since $\beta(G_i)\leq t-1$ and $\Delta(G_i)\leq h-1$, Lemma \ref{matching-delta-function} yields $e(G_i)\leq f(t-1,h-1)$.
Thus $V_i\setminus (W\cup L)$   contains at most  $f(t-1,h-1)$  dominant pairs.
By the definition of $D_i$, we have $|D_i|\leq 2e(G_i)\leq 2f(t-1,h-1)$.
Finally, since $T_i=V_i\setminus (D_i\cup(W\cup L)$, it follows that $|T_i|\geq |V_i\setminus (W\cup L)| - 2f(t-1,h-1).$
\end{proof}

For a vertex $u\in V, X\subseteq V$, we use $E_X(u)$ to denote set of  edges containing $u$ and intersecting $X\setminus \{u\}$ and let  $e_X(u)=|E_X(u)|$.
By Lemma \ref{lem:size-L}, we have
\begin{align} \label{ineq:size-Lu}
	e_L(u)\leq |L|\binom{n-2}{r-2} <\ell:=\varepsilon^{\frac{3}{2r^2}}n^{r-1}.
\end{align}

\begin{lemma}\label{claimA}
For a vertex $u\in V(\mathcal H)\setminus W$ and each $i\in[k]$,
let $E_i(u)$ denote the set of edges containing $u$ and at least one
non-dominant pair in $V_i\setminus( W\cup L\cup\{u\})$.
Then $|E_i(u)| = O(n^{r-2})$.
\end{lemma}
\begin{proof}
For each $i \in [k]$, we construct a graph $G_i$ on vertex set $V_i \setminus (W \cup L)$,  in which
 $v$ and $w$ are adjacent   if $\{v,w\}$ is a non-dominant pair and the triple $\{u,v,w\}$ is contained in at least $h\binom{n-4}{r-4}$ edges of $\mathcal{H}$.

We first show that  in $G_i$,  the number of vertices with degree greater than  $h$ is at most $h-1$.
 Suppose for a contradiction that there exist at least $h$ such vertices. Since $u \notin W$, we may choose an $r$-set $f$ with $u \in f$ and $f \notin E(\mathcal H)$. Define $\mathcal H' = \mathcal H + f$. By the maximality of $\mathcal H$,  $\mathcal H'$ contains a copy of $tK_{k+1}^{(r)}$, denoted $\mathcal F_0$, such that $f\in E(\mathcal F_0)$. Let $\mathcal F_1$ be a subhypergraph of $\mathcal F_0$ isomorphic to $K_{k+1}^{(r)}$ that contains $u$, and set $\mathcal F_2=\mathcal F_0\setminus\mathcal F_1$, so $\mathcal F_2\cong (t-1)K_{k+1}^{(r)}$. Then $u\notin V(\mathcal F_2)$.
 Since $G_i$ has at least $h$ vertices of degree greater than $h$, we may pick a vertex $w\in V(G_i)\setminus V(\mathcal F_2)$ with $d_{G_i}(w)\ge h$. By the definition of $G_i$, there exists a vertex $w'\in V(G_i)\setminus V(\mathcal F_2)$ such that $\{w,w'\}$ is a non-dominant pair and the triple $\{u,w,w'\}$ is contained in at least $h\binom{n-4}{r-4}$ edges of $\mathcal{H}$. By  Lemma~\ref{set}, there exist vertices $w_2,\dots,w_k\notin V(\mathcal F_2)$ such that every pair in  $\{w,w',w_2,\dots,w_k\}$, except for
 $\{w, w'\}$, is a  dense pair and
\[
\partial(\mathcal H[w,w',w_2,\dots,w_k])\cong  K_{k+1}.
\]

Since $\{u,w,w'\}$ is contained in at least $h\binom{n-4}{r-4}$ edges, there exists an edge  $e\in E(\mathcal H)$ such that $\{u,w,w'\}\subseteq e$ and  $e\cap V(\mathcal F_2)=\emptyset$.
We then extend this copy of $K_{k+1}$ to  a copy of $K^{(r)}_{k+1}$, denoted $\mathcal F'$,
satisfying  $V(\mathcal F')\cap V(\mathcal F_2)=\emptyset$. Hence $\mathcal H $ contains a copy of $tK_{k+1}^{(r)}$, a contradiction.
It follows that $G_i$ has at most $h-1$ vertices of degree exceeding $h$. Consequently, $$e(G_i)< \frac12\bigl(h(|V_i|-1) + (|V_i|-h)h\bigr) <  h|V_i|.$$
We thus obtain
$
|E_i(u)|\leq e(G_i)\binom{n-3}{r-3}
        +\Bigl(\binom{|V_i|}{2}-e(G_i)\Bigr)h\binom{n-4}{r-4}
        =O(n^{r-2}).
$
\end{proof}

We next establish an upper bound on degrees of vertices in $L$.
\begin{lemma}\label{lem:degree-vertex-in-L}
For each $u\in L$, we have
\[
    d_{\mathcal H}(u)
    < \binom{k-1}{r-1}\left(\frac{n}{k}\right)^{r-1}-2\ell.
\]
\end{lemma}

\begin{proof}
Without loss of generality, we assume that $u\in V_1$.
We distinguish two cases depending on how many vertices in $V_1\setminus L$ form dominant pairs with~$u$.

\medskip
\noindent\textbf{Case 1.}
There exist at least $h$ vertices in $V_1\setminus L$ whose codegree
with $u$ is at least $d$.

\smallskip
 We first claim that there exists some $j\in\{2,\ldots,k\}$ such that
$|DV_j(u)|\le k\varepsilon^{1/r^{2}} n$.
Otherwise, suppose
$|DV_j(u)|>k\varepsilon^{1/r^{2}} n$ for all $j\in [2, k]$.
Since $u\in L$, there exists an $r$-set $f\subseteq V(\mathcal H)$ containing $u$ with $f\notin E(\mathcal{H})$.
By the maximality of $\mathcal{H}$, the $r$-graph $\mathcal{H}'=\mathcal{H}+f$ contains a copy of $tK_{k+1}^{(r)}$, denoted  $\mathcal F_0$, with $f\in E(\mathcal F_0)$.
Moreover, $\mathcal F_0$ contains a copy of $(t-1)K_{k+1}^{(r)}$, denoted $\mathcal F_1$, such that $u\notin V(\mathcal F_1)$.
By  the assumption of $u$, there exists a vertex $v\in V_1\setminus (W\cup L\cup V(\mathcal F_1)$  such that $\{u, v\}$ is a dominant pair in $\mathcal H$.
By Lemma \ref{capset},
\begin{eqnarray*}
\left|DV_{2}(u)\cap DV_{2}(v)  \setminus L\right|
\geq k\varepsilon^{1/r^2}n + (|V_{2}|-\varepsilon^{1/r^2}n)-|V_{2}|-\varepsilon^{\frac{r-1}{r^2}}n
>h.
\end{eqnarray*}
Thus we can find a vertex $w_2\in V_2\setminus L$ such that $\{u, w_2\}$ and $\{v, w_2\}$ are dense pairs. Continuing this process, we can find a vertex $w_i\in V_i\setminus L$ for each $i\in [2, k]$ such that every pair in $\{u, v, w_2, \ldots, w_k\}$, except for $\{u, v\}$,  is a dense pair.
Thus $\partial (\mathcal H[u, v,w_{2},\ldots,w_{k}])\cong K_{k+1}$, and we may  extend this subgraph to a copy of $K^{(r)}_{k+1}$,  denoted $\mathcal F_2$, with $V(\mathcal F_2)\cap V(\mathcal F_1)=\emptyset$.
Hence $\mathcal F_1\cup \mathcal F_2\cong tK_{k+1}^{(r)}\subseteq\mathcal H$, a contradiction.
Thus, without loss of generality, we may assume
\[
    |DV_2(u)|\le k\varepsilon^{1/r^{2}} n.
\]
To estimate $d_{\mathcal H}(u)$, note that $E_u$ is a subset of the union of the following three sets:

(i) $E_{V_2}(u)$ (the set of edges containing $u$ and intersecting $V_2\setminus\{u\}$);

(ii) $E_{V_1}(u)$ (the set of edges containing $u$ and intersecting $V_1\setminus\{u\}$;

(iii) The set of edges containing $u$ and disjoint from $V_1\cup V_2\setminus\{u\}$.

\noindent
Using Lemmas~\ref{WL}, \ref{independentset}, and~\ref{claimA}, we obtain
\begin{eqnarray}\label{eq4.5}
e_{V_2}(u)&\leq &\big(|L|+t-1\big)\binom{n-2}{r-2}+\left(k f(t{-}1,h{-}1)\binom{n-3}{r-3}+O(n^{r-2})\right)\nonumber \\
& & + |DV_2(u)| \binom{k-1}{r-2}
        \left(\frac{1}{k}+\varepsilon^{1/r}\right)^{r-2}n^{r-2} + \big(|V_2|-|DV_2(u)|\big)d,
\end{eqnarray}
where the first term gives an upper bound on the number of edges in $E_{W\cup L}(u)$; the second term provides the number of edges in $E_u$ that contain a pair vertices in some $V_i\setminus (W\cup L\cup\{u\})$; the third term offers an upper bound on the number of edges in $E_u$ that intersects a vertex in $DV_2(u)$ with $|(e\setminus\{u\})\cap V_i|\le 1$ for all $i\in [k]$; the last term is an upper bound on the number of edges in $E_u$ that intersects a vertex in $V_2\setminus DV_2(u)$ with $|(e\setminus\{u\})\cap V_i|\le 1$ for all $i\in [k]$.

Since $|DV_2(u)\leq 2k\varepsilon^{1/r^{2}} n$, it follows from (\ref{eq4.5}) that
\begin{eqnarray*}
e_{V_2}(u)
    < 2k\varepsilon^{1/r^{2}} n
      \binom{k-1}{r-2}\left(\frac{n}{k}\right)^{r-2}
      + 3\ell.
\end{eqnarray*}

For (ii), we show that $e_{V_1}(u)\leq  2 e_{V_2}(u)$.
Consider the alternative partition
\[
    \bm{\sigma}'=(V_1\setminus\{u\},\, V_2\cup\{u\},\, V_3,\ldots,V_k)
\]
and compare $f_{\mathcal H}(\bm{\sigma}')$ with $f_{\mathcal H}(\bm{\sigma})$.
Edges counted in $e_{V_2}(u)$ may reduce their contribution by~$1$
under $\bm{\sigma}'$, giving a total decrease of at most $e_{V_2}(u)$.
Conversely, edges counted in $e_{V_1}(u)$  disjoint from $V_2$
may increase their contribution by~$1$, giving a total increase of
$e_{V_1}(u)-e_{V_2}(u)$.
Thus,
\[
    f_{\mathcal H}(\bm{\sigma}') - f_{\mathcal H}(\bm{\sigma})
    \ge e_{V_1}(u)-2e_{V_2}(u).
\]
Since $\bm{\sigma}$ is optimal,
\[
    e_{V_1}(u)\le 2e_{V_2}(u).
\]

For (iii), it is easy  to check that the number of edges disjoint from $V_1\cup V_2\setminus \{u\}$ is at most
$$\binom{k-2}{r-1}(\frac{1}{k}+\varepsilon^{1/r})^{r-1}n^{r-1}+3\ell\leq \binom{k-2}{r-1}(\frac{n}{k})^{r-1}+4\ell.$$
So we deduce that
\
\begin{align*}
d_{\mathcal H}(u)
 &\leq 3e_{V_2}(u)+\binom{k-2}{r-1}\left(\frac{n}{k}\right)^{r-1}
        + 4\ell \\
 &< 6k\varepsilon^{1/r^{2}}n
        \binom{k-1}{r-2}\left(\frac{n}{k}\right)^{r-2}
        + \binom{k-2}{r-1}\left(\frac{n}{k}\right)^{r-1}
        + 13\ell \\
 &< \binom{k-1}{r-1}\left(\frac{n}{k}\right)^{r-1}-2\ell,
\end{align*}
where the final inequality follows from
$\binom{k-1}{r-1}-\binom{k-2}{r-1}=\binom{k-2}{r-2}$
and that $\varepsilon$ is sufficiently small.

\medskip
\noindent\textbf{Case 2.}
There are at most $h-1$ vertices in $V_1\setminus L$ that form dominant pairs with $u$.

\smallskip
In this case, the number of edges $e\in E_u$ with  $|V_1\cap e|\geq 2$ is at most
\[
    (h-1)\binom{n-2}{r-2}+(|V_1|-h+1)d+ |L|\binom{n-2}{r-2}< \ell.
\]
By Lemmas~\ref{WL}, \ref{independentset}, and~\ref{claimA}, edges $e\in E_u$ with $|e\cap V_i|\ge 2$ for some $i\in[2,k]$ contribute at most $3\ell$.

Since $u\in L$, $u$ is incident to at least $\varepsilon^{1/r^{2}}n$ sparse pairs.
Each such sparse pair reduces the potential contribution of the complete
$k$-partite $r$-graph by roughly
$\binom{k-2}{r-2}(n/k)^{r-2}-d$.
So the number of edges $e$ of $E_u$ with $|V_i\cap e|\leq 1$ for all $ i\in [k]$ is at most
\begin{eqnarray*}
& &\binom{k-1}{r-1}\left(\frac{n}{k}\right)^{r-1}
 - \frac{\varepsilon^{1/r^{2}}n}{r-1}
   \left[
      \binom{k-2}{r-2}\left(\frac{n}{k}\right)^{r-2}-d
   \right]\\
   &<&\binom{k-1}{r-1}\left(\frac{n}{k}\right)^{r-1}
 - \Omega(\varepsilon^{1/r^{2}})\,n^{r-1}.
 \end{eqnarray*}
Then
\begin{align*}
d_{\mathcal H}(u)
 &<
 \binom{k-1}{r-1}\left(\frac{n}{k}\right)^{r-1}
 - \Omega(\varepsilon^{1/r^{2}})\,n^{r-1}
 + 4\ell \\
 &<
 \binom{k-1}{r-1}\left(\frac{n}{k}\right)^{r-1}
 - 2\ell,
\end{align*}
where the last inequality holds since $\varepsilon^{1/r^{2}}\gg \varepsilon^{3/(2r^{2})}
= \ell/n^{r-1}$.

Combining these two cases completes the proof. \end{proof}

\begin{lemma}\label{lem:x>0-L}
Let $\mathbf{x}$ be the nonnegative eigenvector of $\mathcal H$ corresponding to $\lambda^{(p)}(\mathcal H)$. Then $x_u>0$ for every  non-isolated vertex $u\in V(\mathcal H)$.
\end{lemma}
\begin{proof}
 By Theorem \ref{PFThm}, it suffices to prove that $\mathcal H$ is connected for $p=r$. Suppose for a contradiction that $\mathcal H$ is disconnected, and
 let $\mathcal H_1,\mathcal H_2,\ldots,\mathcal H_s$ be the connected components of $\mathcal H$. Assume $\max\{\lambda^{(r)}(\mathcal H_i) : i\in [s]\}=\lambda^{(r)}(\mathcal H_1)$.
  Then  $\lambda^{(r)}(\mathcal H)=\lambda^{(r)}(\mathcal H_1)$.
 Take an arbitrary vertex $v\in V(\mathcal H_1)$,  let $\mathcal H'$ denote the $r$-graph obtained from $\mathcal H_1$ by attaching a pendent edge $f$ at $v$ and adding
 $n-|V(\mathcal H_1)|-r+1$ isolated vertices.
 Since $\mathcal H_1$ is a subgraph of $\mathcal H'$,
 we have $\lambda^{(r)}(\mathcal H')>\lambda^{(r)}(\mathcal H_1)=\lambda^{(r)}(\mathcal H)$.
Clearly, $\mathcal H'$ is $tK_{k+1}^{(r)}$-free, which contradicts the maximality of $\lambda^{(r)}(\mathcal H)$. Thus $\mathcal H$ must be  connected, as required.
\end{proof}

\begin{lemma}\label{lem:degree-of-u0}
	Let $z$ be a vertex with $x_z = \max\{x_w : w\in V(\mathcal H)\}$ and let $u_0$ be a vertex such that $x_{u_0}=\max\{x_w : w\in V(\mathcal H)\setminus (W\setminus L)\}$.
	Then $x_{u_0}\geq c_0^{\frac{1}{r-1}}x_z$,  where $c_0=\frac{(k)_r}{2k^{r}}$,  and
	\[
	d_{\mathcal H}(u_0) > \binom{k-1}{r-1}\left(\frac{n}{k}\right)^{r-1} - 2\ell.
	\]
\end{lemma}
\begin{proof}
	We first show that $z\notin L$.
	By the eigenvalue-eigenvector equation for $\lambda^{(p)} (\mathcal H)$ at vertex $z$, we have
	\[
	\lambda^{(p)} (\mathcal H)\cdot x_z^{p-1} = (r-1)! \sum_{z\in e, e\in E(\mathcal H)} \mathbf{x}^{e\setminus \{z\}}
	\leq (r-1)! \cdot d_{\mathcal H}(z) x_z^{r-1}.
	\]
	Combining this with (\ref{eq4.1}), it follows that
	\begin{equation}\label{eq:d(z)-lower-bound}
		d_{\mathcal H}(z) \geq \frac{\lambda^{(p)} (\mathcal H)}{(r - 1)!} \cdot x_z^{p-r}
		> \big(1 - O(n^{-1})\big) \cdot \frac{\binom{k-1}{r-1} n^{r(1-1/p)}}{k^{r-1}} \cdot x_z^{p-r}.
	\end{equation}
Since $\|\bm{x}\|_p=1$, we have  $x_z \geq n^{-1/p}$. Combining this with
	 \eqref{eq:d(z)-lower-bound}, we obtain
	\begin{align*}
		d_{\mathcal H}(z) & > \big(1 - O(n^{-1})\big) \cdot \frac{\binom{k-1}{r-1} n^{r-1}}{k^{r-1}} >
		\binom{k-1}{r-1}\left(\frac{n}{k}\right)^{r-1} - 2\ell.
	\end{align*}
 It then follows from  Lemma \ref{lem:degree-vertex-in-L} that $z\notin L$.

Next, we prove that $x_{u_0}\geq c_0^{\frac{1}{r-1}}x_z$.
	From the eigenvalue-eigenvector equation for $\lambda^{(p)} (\mathcal H)$ at vertex $z$, we have
	\begin{align*}
		\lambda^{(p)} (\mathcal H)x_z^{p-1}&=(r-1)!\bigg(\sum_{z\in e,  |(e\setminus\{z\})\cap (W\setminus L)|\geq 1}\mathbf{x}^{e\setminus \{z\}}+\sum_{z\in e, |(e\setminus\{z\})\cap (W\setminus L)|= 0}\mathbf{x}^{e\setminus \{z\}}\bigg)\\[2mm]
		&\leq (r-1)!\bigg(|W\setminus L|\binom{n-2}{r-2}x_z^{r-1}+\binom{n-1}{r-1}x_{u_0}^{r-1}\bigg).
	\end{align*}
	Using asymptotic identity $\binom{n-1}{r-1} = \frac{n^{r-1}}{(r-1)!}\big(1+O(n^{-1})\big)$ and Lemma \ref{WL}, we derive
	\[
	x_{u_0}^{r-1} \geq \frac{\lambda^{(p)}(\mathcal H)}{n^{r-1}}x_z^{p-1} - O(n^{-1})x_z^{r-1}.
	\]
	In conjunction with (\ref{eq4.1}) and the inequality $x_z \geq n^{-1/p}$, we have
	\[
	x_{u_0}^{r-1} \geq (1-O(n^{-1}))\frac{(k)_r}{k^r}x_z^{r-1} - O(n^{-1})x_z^{r-1} \geq c_0 x_z^{r-1},
	\]
	which implies $x_{u_0}\geq c_0^{\frac{1}{r-1}}x_z$.

	We now establish the degree lower bound for $u_0$. From the eigenvalue-eigenvector equation for $\lambda^{(p)} (\mathcal H)$ at $u_0$, we have
	\begin{align}\label{du0}
		\lambda^{(p)} (\mathcal H)x_{u_0}^{p-1}&=(r-1)!\bigg(\sum_{u_0\in e,  |e\cap (W\setminus L)|\geq 1}\mathbf{x}^{e\setminus \{u_0\}}+\sum_{u_0\in e, |e\cap (W\setminus L)|= 0}\mathbf{x}^{e\setminus \{u_0\}}\bigg)\nonumber\\[2mm]
		&\leq (r-1)!\bigg((t-1)\binom{n-2}{r-2}x_z^{r-1}+d_{\mathcal H}(u_0)x_{u_0}^{r-1}\bigg).
	\end{align}
	Using the normalization condition $\|\mathbf{x}\|_p=1 $ and $|W\setminus L|\leq t-1$, we obtain
	\begin{align*}
		1=\sum_{w\in W\setminus L}x_w^{p}+\sum_{w\notin W\setminus L}x_w^{p}
		\leq (t-1)x_z^{p}+(n-t+1)x_{u_0}^{p}.
	\end{align*}
Since $x_{u_0}\geq c_0^{\frac{1}{r-1}}x_z$ and $|W\setminus L|\leq t-1$, there exists a constant $c_1$ such that
	\begin{align}\label{xu0}
	x_{u_0}^{p}\geq \frac{1}{n+c_1}=\frac{1}{n}\bigg(1-\frac{c_1}{n+c_1}\bigg).
	\end{align}
		Combining this with (\ref{eq4.1}),  (\ref{du0}) and rearranging the terms, we obtain
	\begin{align*}
		d_{\mathcal H}(u_0)&\geq \frac{\lambda^{(p)}(\mathcal H)x_{u_0}^{p-r}}{(r-1)!}-O(n^{r-2})\\[2mm]
		&\geq \biggl(1-O\bigg(\frac{1}{n}\bigg)\biggr)\binom{k-1}{r-1}\bigg(\frac{n}{k}\bigg)^{r-1}-O(n^{r-2})\\[2mm]
		&\geq \binom{k-1}{r-1}\bigg(\frac{n}{k}\bigg)^{r-1}-2\ell.
	\end{align*}
	This completes the proof.
\end{proof}

\begin{lemma}\label{lem:L-emptyset}
$L = \emptyset$.
\end{lemma}

\begin{proof}
Suppose for a contradiction that $L\neq\emptyset$. By Lemma~\ref{lem:x>0-L}, we may
choose a vertex $u\in L$ with $x_u > 0$.  Recall that $x_{u_0}=\max\{x_w : w\in V(\mathcal H)\setminus (W\setminus L)\}$;
we assume without loss of generality that $u_0\in V_1$. Furthermore,  Lemmas \ref{lem:degree-of-u0} and
\ref{lem:degree-vertex-in-L} imply that $u_0\notin L$. Let $BE_{u_0}\subseteq E_{u_0}$ denote the subset of edges that are either bad edges or intersect $W\cup L$.
 Note that $BE_{u_0}$ is  contained in the union of the following four sets:

 (i) $BE_{u_0}^1$: set of  edges in $E_{u_0}$ containing two vertices  in
    $V_j\setminus  (W\cup L)$ for some $j\in [k]$ and are disjoint form $W\cup L$;

 (ii) $BE_{u_0}^{2}$: set of edges in $E_{u_0}$ intersecting $V_1\setminus (L\cup \{u_0\}) $ and are disjoint form $W\cup L$;

  (iii) $BE_{u_0}^{3}$: set of  edges in $E_{u_0}$ that intersect $W\setminus L$;

 (iv) $BE_{u_0}^{4}$: set of   edges in $E_{u_0}$  that intersect $L$  and are disjoint from  $W\setminus L$.

\noindent
By Lemmas~\ref{independentset} and~\ref{claimA}, we have
    \begin{eqnarray*}
        |BE_{u_0}^{1}| &\leq& (k-1)f(t-1,h-1)\binom{n-3}{r-3}
        + O(n^{r-2})= O(n^{r-2}),
         \end{eqnarray*}
  and
  \begin{eqnarray*}
      |BE_{u_0}^{2}| &\leq& (|V_1|-h)d
            + (h-1)\binom{n-2}{r-2}= O(n^{r-2}).
    \end{eqnarray*}
Clearly,
 $|BE_{u_0}^{3}|\leq (t-1)\binom{n-2}{r-2}= O(n^{r-2}),$
 and
$|BE_{u_0}^{4}|\leq \ell.$
Recall that $x_{z}=\max\{x_w : w\in V(\mathcal H)\}$ and $x_{u_0}^{r-1}\geq c_0x_z^{r-1}$. We have
\begin{align*}
\sum_{e\in BE_{u_0}} \mathbf{x}^{e\setminus \{u_0\}}
&\leq |BE_{u_0}^{1}|x_{u_0}^{r-1}
    +|BE_{u_0}^{2}|x_{u_0}^{r-1}
    +  |BE_{u_0}^{4}|x_{u_0}^{r-1}
    +|BE_{u_0}^{3}|x_z^{r-1}\\
&< \ell x_{u_0}^{r-1}
   + O(n^{r-2})x_{z}^{r-1}
 < \frac{3}{2}\ell x_{u_0}^{r-1}
\end{align*}
for  sufficiently large $n$. Let $E_{u_0}^c=E_{u_0}\setminus BE_{u_0}$.
Using the eigenvalue-eigenvector equation for $\lambda^{(p)} (\mathcal H)$ at vertex $u_0$, we obtain
\[
\lambda^{(p)} (\mathcal H)\cdot x_{u_0}^{p-1}
= (r-1)! \sum_{e\in E_{u_0}^c} \mathbf{x}^{e\setminus \{u_0\}}
  + (r-1)! \sum_{e\in BE_{u_0}} \mathbf{x}^{e\setminus \{u_0\}}.
\]
Therefore,
\begin{align*}
\sum_{e\in E_{u_0}^c} \mathbf{x}^{e\setminus \{u_0\}}
 &\geq \frac{\lambda^{(p)} (\mathcal H)}{(r-1)!} \cdot x_{u_0}^{p-1}
       - \sum_{e\in BE_{u_0}} \mathbf{x}^{e\setminus \{u_0\}} \\
 &> \bigg ( \frac{\lambda^{(p)} (\mathcal H)}{(r-1)!} \cdot x_{u_0}^{p-r}
            - \frac{3}{2} \ell\bigg) x_{u_0}^{r-1}.
\end{align*}
Combining this with (\ref{eq4.1}) and  (\ref{xu0}), we have
\begin{equation}\label{eq:inequality-last-step}
\sum_{e\in E_{u_0}^c} \mathbf{x}^{e\setminus \{u_0\}}
\geq  \bigg[\binom{k-1}{r-1} \left(\frac{n}{k}\right)^{r-1}
      - 2\ell+\frac{1}{4} \ell\bigg] x_{u_0}^{r-1}.
\end{equation}

Finally, we construct an $n$-vertex $r$-graph $\mathcal H'$  that is
$tK_{k+1}^{(r)}$-free and  has a larger
$p$-spectral radius than $\mathcal H$, thereby yielding a contradiction. To this end, we define
\[
 E_{0}=\{(e\setminus\{{u_0}\})\cup \{u\}: e\in E_{u_0}^c\}.
\]
Since $u\in L$, every $f\in E_{0}$ is an $r$-set of $V(\mathcal H)$ by the definition
of $E_{u_0}^c$.
We now define the  $r$-graph $\mathcal H'$ with  vertex set  $V(\mathcal H') = V(\mathcal H)$ and edge set
$E(\mathcal H') =  E(\mathcal H\setminus \{u\}) \cup E_{0}.$

We first verify that $\mathcal H'$ is $tK_{k+1}^{(r)}$-free. Suppose for a contradiction
that $\mathcal H'$ contains a copy of $tK_{k+1}^{(r)}$, denoted  $\mathcal F_0$.
Then $u\in V(\mathcal F_0)$.
Moreover, $\mathcal F_0$ contains a copy of $K_{k+1}^{(r)}$, say $\mathcal F_1$,
such that $u\in V(\mathcal F_1)$.
Let $\mathcal F_2=\mathcal F_0\setminus \mathcal F_1$. Then
$\mathcal F_2\cong (t-1)K_{k+1}^{(r)}$ and $u\notin V(\mathcal F_2)$.
Let $F$ be the core of $\mathcal F_1$.
 Since $\partial (\mathcal H[E_{u_0}^c])$ is a $k$-partite graph, it follows that
$\partial (\mathcal H'[E_{0}])$ is also $k$-partite, and hence
 $u\not\in F$.
Replacing the edges containing $u$ in $\mathcal F_1$ with their corresponding edges from $E_{u_0}^c$, we obtain a copy of
$K_{k+1}^{(r)}$ in $\mathcal H$ that is disjoint from
$\mathcal F_2$ and does not contain $u$. This implies that $\mathcal H$ contains a copy of $tK_{k+1}^{(r)}$,
a contradiction.

We next show that $\mathcal H'$ has a larger $p$-spectral radius than $\mathcal H$. Note that
\begin{align*}
\lambda^{(p)} (\mathcal H') - \lambda^{(p)} (\mathcal H)
& \geq r! \Bigg( \sum_{e\in E_{0}} \mathbf{x}^{e}- \sum_{e\in E_u} \mathbf{x}^{e} \bigg) \\
& = r! x_u \Bigg( \sum_{e\in E_{u_0}^c} \mathbf{x}^{e\setminus \{u_0\}} - \sum_{e\in E_u} \mathbf{x}^{e\setminus \{u\}}\Bigg).
\end{align*}
By Lemmas \ref{WL}, \ref{lem:degree-vertex-in-L} and \ref{lem:degree-of-u0}, we have
\begin{align*}
\sum_{e\in E_u} \mathbf{x}^{e\setminus \{u\}}
& \leq \bigg(d_{\mathcal H}(u)-(t-1)\binom{n-2}{r-2}\bigg) x_{u_0}^{r-1}+(t-1)\binom{n-2}{r-2}x_z^{r-1}\\[2mm]
& \leq \bigg(\binom{k-1}{r-1} \left(\frac{n}{k}\right)^{r-1}
      - 2\ell+(c_0^{-1}-1)(t-1)\binom{n-2}{r-2}\bigg) x_{u_0}^{r-1}\\[2mm]
& \leq \bigg(\binom{k-1}{r-1} \left(\frac{n}{k}\right)^{r-1}
      - 2\ell+\frac{1}{8} \ell\bigg) x_{u_0}^{r-1}
\end{align*}
Combining this with \eqref{eq:inequality-last-step} , we deduce that
 \[
\lambda^{(p)} (\mathcal H') - \lambda^{(p)} (\mathcal H)
>  \frac{r! \ell}{8} x_u x_z^{r-1} > 0,
\]
which contradicts the maximality of the $p$-spectral radius of $\mathcal H$ among all
$tK_{k+1}^{(r)}$-free $r$-graphs. Hence $L=\emptyset$, completing the proof.
\end{proof}

\begin{lemma}\label{degreeW}
For any vertex $u\in W$, $u$ is adjacent to every other vertex in $V(\mathcal H)$.
\end{lemma}
\begin{proof}
Suppose for a contradiction that there exists a vertex $u\in W$ such that $d(u)<\binom{n-1}{r-1}$. Let $L_{u}=\{e\setminus \{u\}: e\in E(\mathcal H) \text{ and } u\in e\}$. Then there exists at least one $(r-1)$-set $f$ with $f\notin L_{u}$. Let $f'=f\cup \{u\}$ and define $\mathcal H'=\mathcal H+f'$. We claim that $\mathcal H'$ is $tK_{k+1}^{(r)}$-free.
Suppose otherwise that $\mathcal H'$ contains a copy a $tK_{k+1}^{(r)}$, say $\mathcal F$, with $f'\in E(\mathcal F)$. Let $\mathcal F_1\subseteq \mathcal F$ be a copy of  $K_{k+1}^{(r)}$  containing $f'$, and  let $\mathcal F_2$ be the copy of $(t-1)K_{k+1}^{(r)}$ in $\mathcal F$ such that $V(\mathcal F_1)\cap V(\mathcal F_2)=\emptyset$ and $u\notin V(\mathcal F_2)$.
Since $u\in W$, there exists a vertex $v$ such that $\{u,v\}$ is a dominant pair and $v\notin V(\mathcal F_2)$. By Lemmas \ref{set} and \ref{extention}, there exists a copy of $K_{k+1}^{(r)}$, say $\mathcal F_3$, in $\mathcal H$ such that $V(\mathcal F_3)\cap V(\mathcal F_2)=\emptyset$. It follows that $\mathcal F_2\cup \mathcal F_3$ is a copy of $tK_{k+1}^{(r)}$ in  $\mathcal H$, which contradicts the  $tK_{k+1}^{(r)}$-freeness of $\mathcal H$.
\end{proof}

\begin{lemma}\label{lem:down-bound}
For any vertex $v\in V(\mathcal H)\setminus W$, we have
\[
x_v \ge \left(1-\frac{1}{\sqrt{n}}\right)x_{u_0}.
\]
\end{lemma}
\begin{proof}
Suppose for a contradiction that there exists a vertex
$v_0\in V(\mathcal H)\setminus W$ such that $x_{v_0} < \left(1-\frac{1}{\sqrt{n}}\right)x_{u_0}$. Recall that $x_{u_0}=\max\{x_w : w\in V(\mathcal H)\setminus W\}$.
Without loss of generality, we assume that $u_0\in V_1$. Define $E_{\mathcal H\setminus W}(u_0)
= \{e\in E(\mathcal H): u_0\in e,\ e\cap W=\emptyset\}.$
Note that $E_{\mathcal H\setminus W}(u_0)$ is a subset of the union of the following four sets:

(i)  $E_A$: the set of edges $e$ such that $u_0\in e, v_0\notin e$,
\[e\setminus\{u_0\}\subseteq \bigcup_{i=2}^k T_i
\ \text{and} \
|e\cap T_i|\le 1 \ \text{for all } i\in[2,k];\]

(ii) $E_B^{(1)}$: the set of edges $e\in
E_{\mathcal H\setminus W}(u_0)$ that intersect $D$;

(iii) $E_{B}^{(2)}$: the set of edges containing $u_0$ and two vertices from some $T_i$;

(iv)  $E_{C}$: the set of edges containing both $u_0$ and $v_0$.

By Lemmas  \ref{WL}, \ref{independentset} and   \ref{claimA}, we have
\begin{eqnarray*}
	|E_B^{(1)}|
	&\le & \sum_{i=1}^k |D_i|\binom{n-2}{r-2}= O(n^{r-2}), \\
 \end{eqnarray*}
 and
  	\begin{eqnarray*}
 |E_B^{(2)}| &\le&  k f(t-1,h-1)\binom{n-3}{r-3} + (h-1)\binom{n-2}{r-2}+O(n^{r-2})=O(n^{r-2}).
\end{eqnarray*}
Clearly,
\[|E_C| \le \binom{n-2}{r-2} = O(n^{r-2}).
\]
From the eigenvalue-eigenvector equation at $u_0$,  we derive
\[
\lambda^{(p)}(\mathcal H)x_{u_0}^{p-1}
\le
(r-1)!
\left(
	\sum_{e\in E_A}\mathbf{x}^{e\setminus\{u_0\}}
	+
	O(n^{r-2})x_{u_0}^{r-1}
	+
	|W|\binom{n-2}{r-2}x_z^{\,r-1}
\right).
\]
Recalling that $x_{u_0}\ge c_0^{1/(r-1)}x_z$, we obtain
\begin{eqnarray}\label{eq4.10}
\sum_{e\in E_A}\mathbf{x}^{e\setminus\{u_0\}}
\ge
\frac{\lambda^{(p)}(\mathcal H)}{(r-1)!}x_{u_0}^{p-1}
- O(n^{r-2})x_{u_0}^{r-1}.
\end{eqnarray}

Now we construct an $r$-graph $\mathcal H'$ on the same vertex set $V(\mathcal H)$ with
\[
E(\mathcal H')
=
\bigl(E(\mathcal H)\setminus E_{\mathcal H}(v_0)\bigr)
\cup
\{(e\setminus\{u_0\})\cup\{v_0\}: e\in E_A\}.
\]
We first show that $\mathcal H'$ is $tK_{k+1}^{(r)}$-free.
Suppose for a contradiction  that $\mathcal H'$  contains  a copy $\mathcal F_0$ of $tK_{k+1}^{(r)}$ with
$v_0\in V(\mathcal F_0)$. Let $\mathcal F_1$ denote the copy of $K_{k+1}^{(r)}$
in $\mathcal F_0$ that contains $v_0$, and set
$\mathcal F_2 = \mathcal F_0\setminus\mathcal F_1\cong (t-1)K_{k+1}^{(r)}$.
Let $K$ be the core of $\mathcal F_1$.
If $v_0\in K$, then $K\setminus\{v_0\}\subset \bigcup_{i=2}^k T_i$.
By Lemma~\ref{WL}, there exists a vertex
$v'\in T_1\setminus V(\mathcal F_0)$  adjacent via dense pairs
to every vertex in $K\setminus\{v_0\}$.
Replacing $v_0$ with $v'$ in $\mathcal F_1$ yields another copy of
$K_{k+1}^{(r)}$ in $\mathcal H$ that is vertex-disjoint  from $\mathcal F_2$,
which implies  the existence of a copy of $tK_{k+1}^{(r)}$ in $\mathcal H$, a contradiction.
If $v_0\notin K$, let $e$ denote the edge of $\mathcal F_1$
containing $v_0$. Then $e\cap K = \{v_1,v_2\}$ is a dense pair,
so there exists an edge $f\in E(\mathcal H)$ such that
$f\cap V(\mathcal F_0) = \{v_1,v_2\}$.
Replacing $e$ with $f$ in $\mathcal F_0$ again yields a copy of
$tK_{k+1}^{(r)}$ in $\mathcal H$, a contradiction.

We next prove that $\lambda^{(p)}(\mathcal H')>\lambda^{(p)}(\mathcal H)$. From the construction of $\mathcal H'$, we have
\[
\lambda^{(p)}(\mathcal H')-\lambda^{(p)}(\mathcal H)
\ge
r!x_{v_0}
\bigg(
	\sum_{e\in E_A}\mathbf{x}^{e\setminus\{u_0\}}
	-
	\sum_{e\in E_{\mathcal H}(v_0)}\mathbf{x}^{e\setminus\{v_0\}}
\bigg).
\]
Since
$x_{v_0}<\left(1-\frac{1}{\sqrt{n}}\right)x_{u_0}$, it follows that
\begin{eqnarray}\label{eq4.11}
\sum_{e\in E_{\mathcal H}(v_0)}\mathbf{x}^{e\setminus\{v_0\}}
=
\frac{\lambda^{(p)}(\mathcal H)}{(r-1)!}x_{v_0}^{p-1}
<
\frac{\lambda^{(p)}(\mathcal H)}{(r-1)!}
\left(1-\frac{1}{\sqrt{n}}\right)^{p-1}
x_{u_0}^{p-1}.
\end{eqnarray}
Combining (\ref{eq4.10}) and (\ref{eq4.11}), we derive
\begin{align*}
\lambda^{(p)}(\mathcal H') - \lambda^{(p)}(\mathcal H)
&\ge
r!x_{v_0}
\Biggl[
	\frac{\lambda^{(p)}(\mathcal H)}{(r-1)!}x_{u_0}^{p-1}
	-
	O(n^{r-2})x_{u_0}^{r-1}
	-
	\lambda^{(p)}(\mathcal H)
	\left(1-\frac{1}{\sqrt{n}}\right)^{p-1}
	x_{u_0}^{p-1}
\Biggr] \\
&\ge
r!x_{v_0}
\left[
	\frac{p-2}{\sqrt{n}}\lambda^{(p)}(\mathcal H)x_{u_0}^{p-r}
	-
	O(n^{r-2})
\right]
x_{u_0}^{r-1}
> 0,
\end{align*}
where the second inequality hods since $(1-\frac{1}{\sqrt n})^{p-1}\leq 1- \frac{p-2}{\sqrt{n}}$ for sufficiently large $n$.
This contradicts the maximality of $\lambda^{(p)}(\mathcal H)$. \end{proof}

\begin{lemma}\label{lem:W=t-1}
$|W|=t-1$.
\end{lemma}
\begin{proof}
Let $|W|=s$. By Lemmas~\ref{WL} and~\ref{lem:L-emptyset}, we have $s\le t-1$.
Suppose for a contradiction  that $s<t-1$.
Recall that $x_{u_0}=\max\{x_w:w\in V(\mathcal H)\setminus W\}$. Without loss of generality, assume that  $u_0\in V_1$.
Let $E_0$ denote the set of bad edges of $\mathcal H$ that are disjoint from $W$,
and let  $E_1=\{e: e \  \mbox{is an $r$-set containing $u_0$ and}\  e\not\in E(\mathcal H)\setminus E_0\}$.
Let $\mathcal H'$ be obtained from $\mathcal H$ by deleting all edges in $E_0$ and adding all edges in $E_1$.
Then $\mathcal H'$ is a subgraph of $K_{t-1}^{r} \,\vee\, T_r(n-t+1, k)$.
Then $\mathcal H'$ is  $tK_{k+1}^{(r)}$-free.

We next show that $\lambda^{(p)}(\mathcal H')>\lambda^{(p)}(\mathcal H)$, which leads to the desired contradiction.
In fact, we have
\begin{align}\label{ieq:W=t-1}
\lambda^{(p)}(\mathcal H')-\lambda^{(p)}(\mathcal H)
&\ge r!\left(\sum_{e\in E_1}\mathbf{x}^{e}-\sum_{e\in E_0}\mathbf{x}^e\right)\nonumber\\
&\ge r!|E_1|\left(1-\frac{1}{\sqrt{n}}\right)^r x_{u_0}^r
 - r!|E_0|\,x_{u_0}^r.
\end{align}
By the maximality of $\lambda^{(p)}(\mathcal H)$,
$\mathcal H$ contains a copy of $(t-1)K_{k+1}^{(r)}$, denoted $\mathcal F_0$, such that $W\subset V(\mathcal F_0)$.
We claim that every edge $e\in E_0$ intersects $V(\mathcal F_0)$.
Otherwise, there exists a bad edge $e\in E_0$ with
$e\cap V(\mathcal F_0)=\emptyset$, and without loss of generality,
a pair $\{u,v\}\subset e\cap V_1$.
By Lemmas \ref{set} and \ref{extention},
 $\mathcal H$ contains a copy of $K^{(r)}_{k+1}$ containing $e$, denoted $\mathcal F'$, with $V(\mathcal F')\cap V(\mathcal F_0)=\emptyset$.
 As a consequence, $\mathcal F'\cup\mathcal F_0\cong tK_{k+1}^{(r)}$,  which contradicts the fact that $\mathcal H$ is $tK_{k+1}^{(r)}$-free.
We therefore conclude that every edge in $E_0$  intersect $V(\mathcal F_0)$.

We  now estimate the size of $|E_0|$. For a fixed vertex $u\in (V(\mathcal F_0)\cap V_1)\setminus W,$
Lemma \ref{degree} implies that  the number of edges in  $E_0$  containing $u$ and intersecting $V_1\setminus\{u\}$ is at most
 \[
     (|V_1|-h)d + (h-1)\binom{n-2}{r-2}
    = O(n^{r-2}).
    \]
 Moreover, by Lemmas \ref{independentset} and \ref{claimA},  the number of edges in  $E_0$ that contain $u$ and at least two  vertices of $V_j\setminus W$ for some $j\in[k]$
    is also  at most $O(n^{r-2})$.
We thus have
\[
|E_0|
\le (|V(\mathcal F_0)|-s)\cdot O(n^{r-2})
= O(n^{r-2}).
\]
On the other hand, by the definition of $E_1$, it holds that
\[
|E_1|\ge |V_1|\binom{n-s-1}{r-2}=O(n^{r-1}).
\]
Substituting these into \eqref{ieq:W=t-1}, we obtain
\[
\lambda^{(p)}(\mathcal H')>\lambda^{(p)}(\mathcal H),
\]
 for sufficiently large $n$,
which contradicts the maximality of $\lambda^{(p)}(\mathcal H)$.
We therefore conclude that  $s=t-1$.
\end{proof}

\begin{lemma}\label{lem:strong-independent-set}
	For each $i\in[k]$ and every edge $e\in E(\mathcal H\setminus W)$, we have $|e\cap (V_i\setminus W)|\leq 1$.
\end{lemma}
\begin{proof}
	Suppose to the contrary that, without loss of generality, there exists a pair $\{u,v\}\subseteq V_1\setminus W$  contained in some bad edge $e\in E(\mathcal H\setminus W)$.
By Lemmas \ref{set} and \ref{extention}, $\mathcal H$ contains a copy of $K^{(r)}_{k+1}$ containing $e$.
Further,  Lemmas~\ref{degreeW} and~\ref{lem:W=t-1} guarantee the existence of an additional $t-1$  vertex-disjoint copies of $K_{k+1}^{(r)}$, all of which are vertex-disjoint from the aforementioned copy of $K^{(r)}_{k+1}$. This implies that $\mathcal H$ contains a copy of
 $tK_{k+1}^{(r)}$ in $\mathcal H$, a contradiction.
\end{proof}

\noindent \textbf{Proof of Theorem \ref{tH}}.
Let $\mathcal H^*=\mathcal H[V(\mathcal H)\setminus W]$, $U_i=V_i\setminus W$, and $|U_i|=n_i$ for $i\in [k]$.
By Lemmas \ref{lem:L-emptyset}, \ref{degreeW},  \ref{lem:W=t-1}, and \ref{lem:strong-independent-set},
 $\mathcal H$ is a
subgraph of $K_{t-1}^{r}\vee K_k^{r}(n_1,\cdots,n_k)$. Given the maximality of $\lambda^{(p)} (\mathcal H)$, it follows
from Lemma \ref{lem:balance-set}
that $\mathcal H\cong K_{t-1}^{r}\vee T_r(n-t+1,k)$. \hfill $\Box$

\section*{Declaration of interests}
The authors declare that there is no conflict of interest.

\section*{Acknowledgments}The research was partially supported by the National
		Nature Science Foundation of China (grant numbers 12331012, 12571375, 12301437).


\begin{thebibliography}{10}

\bibitem{CioabaFengTaitZhang2020} S. Cioab\u{a}, L.H. Feng, M. Tait, X.D. Zhang, The maximum spectral radius of graphs without friendship subgraphs, \emph{Electron. J. Comb.}, {\bfseries 27}(4) P4.22, 2020.

\bibitem{ChvatalHanson} V. Chv\'atal, D. Hanson, Degree and matchings, \emph{J. Combin. Theory Ser. B}, {\bfseries 20}: 128--138, 1976.



\bibitem{Cooper2012} J. Cooper, A. Dutle, Spectra of uniform hypergraphs, \emph{Linear Algebra Appl.} {\bfseries 436}: 3268--3299, 2012.

\bibitem{EllinghamLuWang2022} M.N. Ellingham, L. Lu, Z. Wang, Maximum spectral radius of outerplanar $3$-uniform hypergraphs, \emph{J. Graph Theory} {\bfseries 100}(4): 671--685, 2022.

\bibitem{FangGaoChangHou2025} L. Fang, G. Gao, A. Chang, Y. Hou, On the spectral Tur\'an problems for bipartite hypergraphs,
     \emph{Electron. J. Comb.}, {\bfseries 32}(4) P4.53, 2025.


\bibitem{Friedman-Wigderson1995} J. Friedman, A. Wigderson, On the second eigenvalue of hypergraphs, \emph{Combinatorica} {\bfseries 15}(1): 43--65, 1995.
	
	


\bibitem{GaoChangHou2022} G. Gao, A. Chang, Y. Hou, Spectral radius on linear $r$-graphs without expanded $K_{r+1}$, \emph{SIAM J. Discrete Math.}, {\bfseries 36}\,(2): 1000--1011, 2022.

   \bibitem{HL1} J. Hou, H. Li, X. Liu, L. Yuan,  Y. Zhang, A step towards a general density
Corr\'adi-Hajnal theorem, \emph{Canad. J. Math.}, 1-36, 2025.

 \bibitem{HL2} J. Hou, C. Hu, H. Li, X. Liu,
C. Yang, Y. Zhang,    Toward a density Corr\'adi-Hajnal theorem for
degenerate hypergraphs, \emph{J. Combin. Theory Ser. B}, {\bfseries 172}: 221--262, 2025.


\bibitem{hou2024criterion} J. Hou, X. Liu, H. Zhao, A criterion for Andr{\'a}sfai-Erd{\H{o}}s-S{\'o}s type theorems and applications,
arXiv preprint arXiv:2401.17219, 2024.
\bibitem{Kang-Nikiforov-Yuan2014}
L. Kang, V. Nikiforov, X. Yuan, The $p$-spectral radius of $k$-partite and $k$-chromatic uniform hypergraphs,
\emph{Linear Algebra Appl.}, {\bfseries 478}: 81--107, 2015.

\bibitem{Keevash2011} P. Keevash, Hypergraph Tur\'an problems,
in Surveys in Combinatorics, Cambridge University Press, Cambridge, pp.83--139, 2011.

\bibitem{Keevash-Lenz-Mubayi2014}
P. Keevash, J. Lenz, D. Mubayi, Spectral extremal problems for hypergraphs,
\emph{SIAM J. Discrete Math.}, {\bfseries 28}\,(4): 1838--1854, 2014.

\bibitem{LiLiuFeng2022} Y. Li, W. Liu, L. Feng, A survey on spectral conditions for some extremal graph problems, \emph{Adv. Math. (China)}, {\bfseries 51(2)}: 193--258, 2022.

\bibitem{LiuNiWangKang2024} L. Liu, Z. Ni, J. Wang, L. Kang, Hypergraph extensions of spectral Tur\'an theorem, preprint available at arXiv: 2408.03122, 2024.

\bibitem{Mubayi2006}
D. Mubayi, A hypergraph extension of Tur\'an's theorem,
\emph{J. Combin. Theory Ser. B}, {\bfseries 96}: 122--134, 2006.

\bibitem{Mubayi-Verstraete2016} D. Mubayi, J. Verstra\"ete,
A survey of Tur\'an problems for expansions, in Recent Trends in Combinatorics,
IMA Vol. Math. Appl. 159, Springer, pp.117--143, 2016.

\bibitem{NiLiuKang2024} Z. Ni, L. Liu, L. Kang, Spectral Tur\'an-type problems on cancellative hypergraphs, \emph{Electron. J. Comb.}, {\bfseries 31(2)} P2.32, 2024.

\bibitem{Nikiforov2014} V. Nikiforov, Analytic methods for uniform hypergraphs, \emph{Linear Algebra Appl.}, {\bfseries 457}: 455--535, 2014.

\bibitem{Nikiforov2007} V. Nikiforov, Bounds on graph eigenvalues II, \emph{Linear Algebra Appl.}, {\bfseries 427}: 183--189, 2007.

\bibitem{Nikiforov2010} V. Nikiforov, The spectral radius of graphs without paths and cycles of specified length, \emph{Linear Algebra Appl.}, {\bfseries 432(9)}: 2243--2256, 2010.

\bibitem{Pikhurko2013} O. Pikhurko, Exact computation of the hypergraph Tur\'an function for expanded complete $2$-graphs, \emph{J. Combin. Theory Ser. B}, {\bfseries 103}: 220--225, 2013.

\bibitem{RodlSkokan2006} V. R{\"o}dl, J. Skokan, Applications of the regularity lemma for uniform hypergraphs, \emph{Random Structures Algorithms}, {\bfseries 28(2)}: 180--194, 2006.

\bibitem{SheFanKangHou2023} C. She, Y. Fan, L. Kang, Y. Hou, Linear spectral Tur\'an problems for expansions of graphs with given chromatic number, \emph{Acta Math. Appl. Sin. Engl. Ser.},  online,  2025.

\bibitem{ZhengLiFan2024} J. Zheng, H. Li, Y. Fan, Spectral Tur\'an problems for hypergraphs with bipartite or multipartite pattern,  preprint available at arXiv: 2409.17678, 2024.
\end{thebibliography}
\end{document}